\documentclass[11pt]{amsart}
\usepackage{palatino}
\usepackage{amsfonts}
\usepackage{amssymb}
\usepackage{amscd}
\usepackage[breaklinks,bookmarksopen,bookmarksnumbered]{hyperref}

\newtheorem{theorem}{Theorem}[section]
\newtheorem{proposition}[theorem]{Proposition}
\newtheorem{corollary}[theorem]{Corollary}
\newtheorem{lemma}[theorem]{Lemma}
\theoremstyle{definition}
\newtheorem{definition}[theorem]{Definition}
\newtheorem{remark}[theorem]{Remark}

\newtheorem{conjecture/question}[theorem]{Conjecture/Question}
\newcommand{\dblq}{{/\!/}}

\newtheorem{remark/definition}[theorem]{Remark/Definition}
\newtheorem{terminology/notation}[theorem]{Terminology/Notation}

\def\DD{{\mathbb D}}

\def\PP{{\textbf P}}
\def\OO{\mathcal{O}}

\def\cB{\mathcal{B}}
\def\cA{\mathcal{A}}
\def\F{\mathcal{F}}
\def\P{\mathcal{P}}

\def\E{\mathcal{E}}
\def\G{\mathcal{G}}
\def\cS{\mathcal{S}}
\def\K{\mathcal{K}}

\def\I{\mathcal{I}}

\def\cM{\mathcal{M}}

\def\cZ{\mathcal{Z}}
\def\cU{\mathcal{U}}
\def\cC{\mathcal{C}}
\def\H{\mathcal{H}}
\def\Pic0{{\rm Pic}^0(X)}

\def\mm{\overline{\mathcal{M}}}
\def\hh{\overline{\mathcal{H}}}
\def\ss{\overline{\mathcal{S}}}
\def\zz{\overline{\mathcal{Z}}}

\def\ps{\widetilde{\mathcal{S}}}

\def\ts{\overline{\textbf{S}}}
\def\thet{\Theta_{\mathrm{null}}}
\def\pem{\widetilde{\textbf{M}}}
\def\pes{\widetilde{\textbf{S}}}
\def\rem{\overline{\textbf{M}}}

\def\tem{\textbf{{M}}^p}
\def\ttem{\overline{\textbf{M}}^p}

\begin{document}
\title{The geometry of the moduli space of odd spin curves}
\author[G. Farkas]{Gavril Farkas}

\address{Humboldt-Universit\"at zu Berlin, Institut f\"ur Mathematik,  Unter den Linden 6
\hfill \newline\texttt{}
 \indent 10099 Berlin, Germany} \email{{\tt farkas@math.hu-berlin.de}}
\thanks{}

\author[A. Verra]{Alessandro Verra}
\address{Universit\'a Roma Tre, Dipartimento di Matematica, Largo San Leonardo Murialdo \hfill
\indent 1-00146 Roma, Italy}
 \email{{\tt
verra@mat.uniroma3.it}}

\maketitle

The set of odd theta-characteristics on a general curve $C$ of genus $g$ is in bijection with the
set $\theta(C)$ of \emph{theta hyperplanes} $H\in (\PP^{g-1})^{\vee}$ everywhere tangent to the canonically embedded curve $C\stackrel{|K_C|}\hookrightarrow \PP^{g-1}$. Even though the geometry and the intricate combinatorics
of $\theta(C)$  have been studied classically, see \cite{Dol}, \cite{DK} for a modern account, it has only been recently proved in \cite{CS} that one can reconstruct a general curve $[C]\in \cM_g$ from the hyperplane configuration $\theta(C)$.

Odd theta-characteristics form a moduli space $\pi:\cS_g^-\rightarrow \cM_g$. At the level of stacks, $\pi$ is an \'etale cover of degree $2^{g-1}(2^g-1)$. The normalization of $\mm_g$ in the function field of $\cS_g^-$ gives rise to a finite covering $\pi:\ss_g^-\rightarrow \mm_g$. Furthermore, $\ss_g^-$ has a modular meaning being isomorphic to the coarse moduli space of the Deligne-Mumford stack of odd stable spin curves, cf. \cite{C}, \cite{CCC}, \cite{AJ}. The
map $\pi$ is branched along the boundary of $\mm_g$ and one expects $K_{\ss_g^-}$ to enjoy better positivity properties than $K_{\mm_g}$.

The aim of this paper is to describe the birational geometry of $\ss_g^-$ for all $g$. Our goals are (1) to understand the transition from rationality to maximal Kodaira dimension for $\ss_g^-$ as $g$ increases, and (2) to use the existence of \emph{Mukai models} of $\mm_g$ in order to construct explicit unirational parameterizations of $\ss_g^-$ for small genus. Remarkably, we end up having no gaps in the classification of $\ss_g^-$. First, we show that in the range where the general curve $[C]\in \cM_g$
lies on a $K3$ surface, the existence of special \emph{theta pencils} on $K3$ surfaces, provides an \emph{explicit} uniruled
parameterization of $\ss_g^-$:
\begin{theorem}\label{uniruled}
The odd spin moduli space $\ss_g^{-}$ is uniruled for $g\leq 11$.
\end{theorem}
We fix a general spin curve $[C, \eta]\in \cS_g^-$, therefore $h^0(C, \eta)=1$. When $g\leq 9$ or $g=11$, the underlying curve $C$ is the hyperplane section of a $K3$ surface $X\subset \PP^g$, such that if $d\in C_{g-1}$ is the (unique) effective divisor with $\eta=\OO_C(d)$, then the linear span $\langle d\rangle\subset \PP^g$ is a codimension two linear subspace. A rational curve  $P\subset \ss_g^-$ is induced by the pencil $\PP H^0(X, \mathcal{I}_{d/X}(C))$ of hyperplanes containing $\langle d\rangle$. We show in Section 3 that $P\subset \ss_g^-$ is  a \emph{covering} rational curve, satisfying $$P\cdot K_{\ss_g^-}=2g-24<0.$$
Thus $P\cdot K_{\ss_g^-}<0$ precisely when $g\leq 11$, which highlights the fact that the nature of $\ss_g^-$ is expected to change exactly when $g\geq 12$. This is something we shall achieve in the course of proving Theorem \ref{kodaira}.

The previous argument no longer works for $\ss_{10}^-$, when the condition that a curve $[C]\in \mm_{10}$ lie on a $K3$ surface is divisorial in moduli \cite{FP}. This case is a specialization of the genus $11$ case. A  general one-nodal irreducible curve $[C]\in \Delta_0\subset \mm_{11}$ of arithmetic genus $11$ lies on a $K3$ surface $X\subset \PP^{11}$.  By a degeneration argument, we show that this construction can also be carried out in such a way, that if $\nu:C'\rightarrow C$ denotes the normalization of $C$, then the points $x, y\in C'$ with $\nu(x)=\nu(y)$ (that is, mapping to the node of $C$), lie in the support of the zero locus of one of the odd-theta characteristics of $[C']\in \cM_{10}$. Ultimately, this produces a rational curve $P\subset \ss_{10}^-$ through a general point, which shows that $\ss_{10}^-$ is uniruled as well.
\vskip 2pt

In the range in which a \emph{Mukai model} of $\mm_g$ exists, our results are more precise:

\begin{theorem}\label{unirational}
$\ss_g^-$ is unirational for $g\leq 8$.
\end{theorem}
The proof relies on the existence, in this range, of \emph{Mukai varieties} $V_g\subset \PP^{n_g+g-2}$, where $n_g=\mbox{dim}(V_g)$, which have the property that general $1$-dimensional linear sections of $V_g$ are canonical curves $[C]\in \cM_g$ with general moduli. We fix an integer $1\leq \delta\leq g-1$ and consider the correspondence
$$\P_{g, \delta}^o:=\Bigl\{(C, \Gamma, Z): Z\subset C\cap \Gamma \subset V_g, \ \  |\mathrm{sing}(\Gamma)|=\delta, \  \mathrm{sing}(\Gamma)\subset Z\Bigr\},$$
where $Z\subset V_g$ is a \emph{cluster}, that is, a $0$-dimensional subscheme of $V_g$ of length $2g-2$, supported at $g-1$ points and such that $\mathrm{dim}\langle Z\rangle =g-2$ (see Section 3 for a precise definition), $\Gamma\subset V_g$ is an irreducible  $\delta$-nodal curve section of $V_g$ whose nodes are among the points in the support of $Z$, and $C\subset V_g$ is an arbitrary curve linear section of $V_g$ containing $Z$ as a subscheme. Thus if $C$ is smooth, then $Z\subset C$ is a divisor of even degree at each point in its support, and $\OO_C(Z/2)$ can be viewed as an odd theta-characteristic. The quotient variety $\overline{\mathbb P}_{g, \delta}:=\P_{g, \delta}^o\dblq\mbox{Aut}(V_g)$ comes equipped with two projections
$$\ss_g^-\stackrel{\overline{\alpha}}\longleftarrow \overline{\mathbb P}_{g, \delta}\stackrel{\overline{\beta}}\longrightarrow B_{g, \delta}^-,$$
where $B_{g, \delta}^-\subset \ss_g^-$ denotes the moduli space of irreducible $\delta$-nodal curves of arithmetic genus $g$ together with an odd theta- characteristic on the normalization. It is easy to see that $\overline{\mathbb P}_{g, \delta}$ is birational to a projective bundle over the irreducible variety $B_{g, \delta}^-$. Thus the unirationality of $\ss_g^-$ follows once we prove that (i) $\alpha$ is dominant, and (ii) $B_{g, \delta}^-$ itself is unirational.
We carry out this program when $g\leq 8$. When $\delta=n_g-1$, we show in Section 3 that the map $\overline{\alpha}$ is birational, hence in this case $\overline{\beta}$ realizes a birational isomorphism between $\ss_g^-$ and a (Zariski trivial) projective bundle over $B_{g, n_g-1}^-$. Very interesting is the case $g=8$, when $n_g=8$, see \cite{M1}, and $B_{8, 7}^-$ is isomorphic to the moduli space $\mm_{1,14}/\mathbb Z_2^{\oplus 7}$ of elliptic curves with seven pairs of points; here each copy of $\mathbb Z_2$ identifies a pair of points.

\begin{theorem}\label{nyolc}
$\ss_8^-$ is birational to $\PP^7\times \bigl(\mm_{1, 14}/\mathbb Z_2^{\oplus 7}\bigr)$.
\end{theorem}

In the process of proving Theorem \ref{unirational}, we establish some facts of independent interest concerning the Mukai models
$$\mathfrak{M}_g:=\textbf G(g-1, n_g+g-2)^{\mathrm{ss}}\dblq \mathrm{Aut}(V_g).$$
These are birational models of $\mm_g$ having $\mathrm{Pic}(\mathfrak M_g)=\mathbb Z$ and appearing as GIT quotients of Grassmannians; they can be viewed as log-minimal models of $\mm_g$ emerging from the constructions carried out in \cite{M1}, \cite{M2}, \cite{M3}.
\vskip3pt

Theorem \ref{uniruled} is sharp and the remaining moduli spaces $\ss_g^-$ are of general type:

\begin{theorem}\label{kodaira}
The space $\ss_g^-$ is a variety of general type for $g>11$.
\end{theorem}
The border case of $\ss_{12}^-$ is particularly challenging and takes up the entire Section 6. We remark that in the range $11<g<17$, of the two moduli spaces $\ss_g^-$ and $\mm_g$, one is of general type whereas the other has negative Kodaira dimension. More strikingly,   Theorems \ref{kodaira} and \ref{uniruled} coupled with results from \cite{F3}, show that for $9\leq g\leq 11$, the space $\ss_g^-$ is uniruled while $\ss_g^+$ is of general type! Finally, we note that $\ss_8^-$ is unirational whereas $\ss_8^+$ is of Calabi-Yau type \cite{FV}.

We describe the main steps in the proof of Theorem \ref{kodaira}. First, we use that for all $g\geq 4$ and $\ell\geq 0$, if $\epsilon: \widehat{\cS}_g\rightarrow \ss_g^{-}$ denotes a resolution of singularities, then there is an induced isomorphism at the level of global sections , see \cite{Lud}
$$\epsilon^*: H^0\Bigl(\ss_{g, \mathrm{reg}}^-, K_{\ss_g^-}^{\otimes \ell}\Bigr)\stackrel{\sim}\longrightarrow H^0\Bigl(\widehat{\cS}_g, K_{\widehat{S}_g}^{\otimes \ell}\Bigr).$$
Thus to conclude that $\ss_g^-$ is of general type, it suffices to exhibit an effective divisor $D$ on $\ss_g^-$ such that for appropriately chosen rational constants $\alpha, \beta>0$, a relation of the type
$K_{\ss_g^-}\equiv \alpha\ \lambda+\beta\ D+E\in \mbox{Pic}(\ss_g^-)$ holds, where $\lambda\in \mbox{Pic}(\ss_g^-)$ is the pull-back to $\ss_g^-$ of the Hodge class, and $E$ is an effective $\mathbb Q$-class which is typically a combination of boundary divisors. It is essential to pick $D$ so that (i) its class can be explicitly computed, that is, points in $D$ have good geometric characterization, and (ii) $[D]\in \mathrm{Pic}(\ss_g^-)$ is in some way an extremal point of the effective cone of divisors so that the coefficients $\alpha, \beta$ stand a chance of being positive. In the case of $\ss_g^+$, the role
of $D$ is played by the divisor $\overline{\Theta}_{\mathrm{null}}$ of vanishing theta-nulls, see \cite{F3}. In the case of $\ss_g^-$ we compute the class of \emph{degenerate theta-characteristics}, that is, curves carrying a non-reduced odd theta-characteristic.
\begin{theorem}\label{degen}
We fix $g\geq 3$. The locus consisting of odd spin curves
$$\cZ_g:=\Bigl\{[C, \eta]\in \cS_g^-:\eta=\OO_C(2x_1+x_2+\cdots+x_{g-2})\ \mbox{ where } x_i\in C \mbox{ for } i=1, \ldots,  g-2 \Bigr\}$$
is a divisor on $\cS_g^-$. The class of its compactification inside $\ss_g^-$ equals
$$\zz_g \equiv (g+8)\lambda-\frac{g+2}{4}\alpha_0-2\beta_0-\sum_{i=1}^{[g/2]} 2(g-i)\ \alpha_i-\sum_{i=1}^{[g/2]} 2i\ \beta_i\in \mathrm{Pic}(\ss_g^-),$$
where $\lambda, \alpha_0, \beta_0, \ldots, \alpha_{[g/2]}, \beta_{[g/2]}$ are the standard generators of $\mathrm{Pic}(\ss_g^-)$.
\end{theorem}
For low genus, $\cZ_g$ specializes to well-known geometric loci. For instance $\cZ_3$ is the divisor of hyperflexes on plane quartics. In particular,  Theorem \ref{degen} yields the formula
$$\pi_*(\zz_3)\equiv 308\lambda-32\delta_0-76\delta_1\in \mathrm{Pic}(\mm_3),$$
for the class of quartic curves having a hyperflex. This matches \cite{Cu} formula (5.5). Moreover, one has the following relation in $\mathrm{Pic}(\mm_3)$
$$\Bigl[\{[C]\in \cM_3: \exists x\in C \mbox{ with } 4x\equiv K_C\}\Bigr]^{-}\equiv 8\cdot \mm_{3, 2}^1 +\pi_*(\zz_3),$$
where $\mm_{3, 2}^1\equiv 9\lambda-\delta_0-3\delta_1$ is the hyperelliptic class and the multiplicity $8$ accounts for the number of Weierstrass points.

We briefly explain how Theorem \ref{degen} implies that $\ss_g^-$ is of general type for $g>11$. We choose an effective divisor $D\in \mbox{Eff}(\mm_g)$ of small slope; for composite $g+1$ one can take $D=\mm_{g, d}^r$ the closure of the Brill-Noether
divisor of curves with a $\mathfrak g^r_d$, where $\rho(g, r, d)=-1$; there exists a constant $c_{g, d, r}>0$ such that \cite{EH2},
$$\mm_{g, d}^r\equiv c_{g, d, r}\Bigl((g+3)\lambda-\frac{g+1}{6}\delta_0-\sum_{i=1}^{[g/2]} i(g-i)\delta_i\Bigr)\in \mathrm{Pic}(\mm_g).$$
We form the linear combination of divisors on $\ss_g^-$
$$\frac{2}{g-2}\zz_g+\frac{3(3g-10)}{c_{g, d, r}(g-2)(g+1)}\pi^*(\mm_{g, d}^r)\equiv \frac{11g+37}{g+1}\lambda-2\alpha_0-3\beta_0-\sum_{i=1}^{[g/2]} (a_i\cdot \alpha_i+ b_i\cdot \beta_i),$$
where $a_i, b_i\geq 2$ for $i\neq 1$ and $a_1, b_1>3$ are explicitly known rational constants. The canonical class of $\ss_g^-$ is given by the Riemann-Hurwitz formula
$$K_{\ss_g^{-}}\equiv\pi^*(K_{\mm_g})+\beta_0 \equiv
13\lambda-2\alpha_0-3\beta_0-2\sum_{i=1}^{[g/2]}
(\alpha_i+\beta_i)-(\alpha_1+\beta_1),$$ and by comparison, it follows that for $g>12$ one can find a constant $\mu_g\in \mathbb Q_{>0}$ such that
$$K_{\ss_g^-}-\mu_g\cdot \lambda\in \mathbb Q_{\geq 0}\langle [\zz_g], \alpha_1,\ \beta_1, \ldots, \alpha_{[g/2]}, \beta_{[g/2]}\rangle,$$ which shows that $K_{\ss_g^-}$ is big and thus proves
Theorem \ref{kodaira}.

\vskip 3pt
For $g=12$, there is no Brill-Noether divisor, and the reasoning above shows that in order to conclude that $\ss_{12}^-$ is of general type, one needs an effective divisor $\overline{\mathfrak{D}}_{12}$ of slope $s(\overline{\mathfrak{D}}_{12})<6+\frac{12}{13}$, that is, a counterexample to the Slope Conjecture on effective divisors on $\mm_{12}$, see \cite{FP}. We define the locus
$$\mathfrak{D}_{12}:=\Bigl\{[C]\in \cM_{12}:\exists L\in W^4_{14}(C) \mbox{ with }\  \mathrm{Sym}^2 H^0(C, L)\stackrel{\mu_0(L)}\longrightarrow H^0(C, L^{\otimes 2}) \ \mbox{  not injective}\Bigr\},$$
that is, points in $\mathfrak{D}_{12}$ correspond to curves that admit an embedding $C\subset \PP^4$ with $\mathrm{deg}(C)=14$, such that $H^0(\PP^4, \mathcal{I}_{C/\PP^4}(2))\neq 0$. The computation of the class of the closure  $\overline{\mathfrak{D}}_{12}\subset \mm_{12}$ is carried out in Section 6 and it turns out that $s(\overline{\mathfrak{D}}_{12})=\frac{4415}{642}<6+\frac{12}{13}$. In particular $\mathfrak{D}_{12}$ violates the Slope Conjecture on $\mm_{12}$, and as such, it contains the locus
$\mathcal{K}_{12}:=\{[C]\in \cM_{12}: C \mbox{ lies on a } K3 \mbox{ surface}\}$.

\vskip 4pt
We discuss the structure of the paper. Section 1 is of preliminary nature and establishes basic facts about the moduli space $\ss_g^-$ which will be used both in Section 3 in the course of proving Theorem \ref{unirational}, as well as in Section 5, when calculating the class $[\zz_g]$. In Section 2, we prove Theorem \ref{uniruled}, whereas Section 3 is devoted to the construction of Mukai models for $\ss_g^-$ and to establishing Theorem \ref{unirational}. The proof of Theorems \ref{kodaira} for $g>12$ is completed in Section 5. Finally, in Section 6, we construct two counterexamples to the Slope Conjecture on $\mm_{12}$, which implies that $\ss_{12}^-$ is of general type.

\section{Families of stable spin curves}
We briefly review some relevant facts about the moduli space $\ss_g^-$ that will be used throughout the paper, see also \cite{C}, \cite{F3}, \cite{Lud} for details. As a matter of notation, we follow the convention set in \cite{FL}; if $\textbf{M}$ is a Deligne-Mumford stack, then we denote
by $\mathcal{M}$ its associated coarse moduli space. Sligthly abusing notation, if $C$ is a smooth curve of genus $g$ and $\eta\in \mbox{Pic}^{g-1}(C)$ an isolated odd theta-characteristic, that is, satisfying $h^0(C, \eta)=1$, we define the \emph{support}  $\mbox{supp}(\eta):=\mbox{supp}(D)$, where $D\in C_{g-1}$ is the unique effective divisor with $\eta=\OO_C(D)$.
An isolated theta-characteristic $\eta$ is said to be non-reduced if $\mbox{supp}(\eta)$  is a non-reduced divisor on $C$.

A connected, nodal curve $X$ is called \emph{quasi-stable}, if for any component $E\subset X$ which is isomorphic to $\PP^1$, one has that (i) $k_E:=|E\cap \overline{(X-E)}|\geq 2$, and (ii) any two rational components $E, E'\subset X$ with $k_E=k_{E'}\geq 2$ are disjoint. Such irreducible components are called \emph{exceptional}. We recall the following definition from \cite{C}:

\begin{definition}\label{stspin}
 A stable \emph{spin curve} of genus
$g$ consists of a triple $(X, \eta, \beta)$, where $X$ is a genus
$g$ quasi-stable curve, $\eta\in \mathrm{Pic}^{g-1}(X)$ is a line
bundle of total degree $g-1$ with $\eta_{E}=\OO_E(1)$ for all
exceptional components $E\subset X$, and $\beta:\eta^{\otimes
2}\rightarrow \omega_X$ is a homomorphism of sheaves which is generically
non-zero along each non-exceptional component of $X$.
\end{definition}
Sometimes the morphism $\beta\in \PP H^0(X, \omega_X\otimes \eta^{\otimes (-2)})$ appearing in the Definition \ref{stspin} is uniquely determined by $X$ and $\eta$ and is accordingly dropped from the notation. In such a case, to ease notation, we view spin curves as pairs $[X, \eta]\in \ss_g$.
It follows from the definition that if  $(X, \eta, \beta)$ is a spin curve with exceptional components
$E_1, \ldots, E_r$  and $\{p_i, q_i\}=E_i\cap \overline{(X-E_i)}$ for
$i=1, \ldots, r$, then $\beta_{E_i}=0$. Moreover,  if
$\widetilde{X}:=\overline{X-\bigcup_{i=1}^r E_i}$ (viewed as a subcurve
of $X$), then we have an isomorphism of sheaves
$
\eta^{\otimes 2}_{\widetilde{X}}\stackrel{\sim}\rightarrow
\omega_{\widetilde{X}}.
$

We denote by $\overline{\textbf{S}}_g$ the non-singular Deligne-Mumford stack of spin curves of genus $g$. Because the parity $h^0(X, \eta) \mbox{ mod 2}$ of a spin curve is invariant under deformations \cite{Mu}, the stack
$\overline{\textbf{S}}_g$  splits into two connected components $\ts_g^+$ and $\ts_g^-$ of relative degree $2^{g-1}(2^g+1)$ and
$2^{g-1}(2^g-1)$ respectively.
It is proved in \cite{C} that the coarse moduli space of $\ts_g$ is isomorphic to the normalization of $\mm_g$ in the function field
of $\cS_g$. There is a proper morphism $\pi:\ss_g\rightarrow \mm_g$ given by $\pi([X, \eta, \beta]):=[\mathrm{st}(X)]$, where $\mathrm{st}(X)$ denotes the stable model of $X$ obtained by contracting all exceptional components.

\subsection{Spin curves of compact type.} We recall the description of the pull-back divisors
$\pi^*(\Delta_i)$ for $1\leq i\leq [g/2]$. We choose a spin curve $[X, \eta, \beta]\in \pi^{-1}([C\cup_y D])$,
where $[C, y]\in \cM_{i, 1}$ and $[D, y]\in \cM_{g-i, 1}$. Then necessarily
$X:=C\cup_{y_1} E\cup_{y_2} D$, where $E$ is an exceptional
component such that $C\cap E=\{y_1\}$ and $D\cap E=\{y_2\}$.
Moreover $\eta=\bigl(\eta_C, \eta_D, \eta_E=\OO_E(1)\bigr)\in
\mbox{Pic}^{g-1}(X)$. Since $\beta_{E}=0$, it follows that  $\eta_C^{\otimes 2}=K_C$ and $\eta_D^{\otimes
2}=K_D$, that is, $\eta_C$ and $\eta_D$ are "honest" theta-characteristics on $C$ and $D$ respectively.
The condition $h^0(X, \eta)\equiv  1 \mbox{ mod } 2$
implies that $\eta_C$ and $\eta_D$ must have
opposite parities. We denote by $A_i\subset \ss_g^-$ the closure in $\ss_g^-$ of
the locus corresponding to pairs $$([C, \eta_C, y], [D, \eta_D, y])\in
\cS_{i, 1}^-\times \cS_{g-i, 1}^+,$$ and by $B_i\subset \ss_g^-$ the
closure in $\ss_g^-$ of the locus corresponding to pairs $$([C, \eta_C, y], [D,
\eta_D, y])\in \cS_{i, 1}^+\times \cS_{g-i, 1}^{-}.$$
One has the relation $\pi^*(\Delta_i)=A_i+B_i$ and clearly $\mbox{deg}(A_i/\Delta_i)=2^{g-2}(2^i-1)(2^{g-i}+1)$
and $\mbox{deg}(B_i/\Delta_i)=2^{g-2}(2^i+1)(2^{g-i}-1)$. One denotes $\alpha_i:=[A_i], \beta_i:=[B_i]\in
\mathrm{Pic}(\ss_g^-)$.

\subsection{Spin curves with an irreducible stable model.} In order to describe $\pi^*(\Delta_0)$ we pick a  point $[X, \eta, \beta]$ such that $\mbox{st}(X)=C_{yq}:=C/y\sim q$,
where $[C, y, q]\in \cM_{g-1, 2}$ is a general point of $\Delta_0$. Unlike the case of curves of compact type, here there are two possibilities
depending on whether $X$ possesses an exceptional component or not.
If $X=C_{yq}$ and $\eta_C:=\nu^*(\eta)$ where $\nu:C\rightarrow X$
denotes the normalization map, then $\eta_C^{\otimes 2}=K_C(y+q)$.
For each choice of $\eta_C\in \mathrm{Pic}^{g-1}(C)$ as above, there
is precisely one choice of gluing the fibres $\eta_C(y)$ and
$\eta_C(q)$ such that $h^0(X, \eta)\equiv 1\mbox{ mod }2$. We denote by $A_0$ the
closure in $\ss_g^-$ of the locus of those points $[C_{yq}, \eta_C\in
\sqrt{K_C(y+q)}]$ with $\eta_C(y)$ and $\eta_C(q)$ glued as above. One has that
$\mbox{deg}(A_0/\Delta_0)=2^{2g-2}$.

If $X=C\cup_{\{y, q\}} E$ where $E$ is an exceptional component, then since $\beta_{E}=0$
it follows that $\beta_{C}\in H^0(C, \omega_{X| C}\otimes \eta_C^{\otimes (-2)})$ must vanish at both
$y$ and $q$ and then for degree reasons $\eta_C:=\eta\otimes \OO_C$ is a theta-characteristic on $C$.
The condition $H^0(X, \omega)\cong H^0(C, \omega_C)\equiv 1 \mbox{ mod 2}$ implies that $[C,
\eta_C]\in \cS_{g-1}^{-}$. In an \'etale neighborhood of a point $[X, \eta, \beta]$,
the covering $\pi$ is given
by $$(\tau_1, \tau_2, \ldots, \tau_{3g-3})\mapsto (\tau_1^2, \tau_2, \ldots, \tau_{3g-3}),$$ where one identifies  $\mathbb C_{\tau}^{3g-3}$
with the versal deformation space of $(X, \eta, \beta)$ and the hyperplane $(\tau_1=0)\subset \mathbb C_{\tau}^{3g-3}$ denotes the locus of spin curves where the exceptional component $E$ persists. This discussion shows that $\pi$ is simply branched over $\Delta_0$ and we denote the ramification divisor by
$B_0\subset \ss_g^-$, that is, the closure of the locus of spin curves
$[C\cup_{\{y, q\}} E, (C, \eta_C)\in \cS_{g-1}^-,\ \eta_E=\OO_E(1)]$. If
$\alpha_0=[A_0]\in \mbox{Pic}(\ss_g^-)$ and $\beta_0=[B_0]\in
\mbox{Pic}(\ss_g^-)$,  we then have the relation
\begin{equation}\label{del0}
\pi^*(\delta_0)=\alpha_0+2\beta_0.
\end{equation}

We define several test curves in the boundary of $\ss_g^-$ which will be later used to compute divisor classes
on the moduli space.

\subsection{The family $F_i$.} We fix $1\leq i\leq [g/2]$ and construct a covering family for the boundary divisor $A_i$.
We fix general  curves $[C]\in \cM_i$ and $[D, q]\in \cM_{g-i, 1}$ as well as an odd theta-characteristic $\eta_C^-$ on $C$ and
an even theta-characteristic $\eta_D^+$ on $D$. If $E\cong \PP^1$ is a fixed exceptional component, we define the family of spin curves
$$F_i:=\Bigl\{[C\cup_y \cup E\cup _q D, \eta]: \eta_C=\eta_C^-, \eta_E=\OO_E(1), \eta_D=\eta_D^+, E\cap C=\{y\}, E\cap D=\{q\}\Bigr\}_{y\in C}.$$
One has that $F_i\cdot \beta_i=0$ and then $F_i\cdot \alpha_i=-2i+2$; furthermore $F_i$ has intersection number zero with the remaining generators of
$\mbox{Pic}(\ss_g^-)$.

\subsection{The family $G_i$.}
As above, we fix an integer $1\leq i\leq [g/2]$ and curves $[C]\in \cM_i$ and $[D, q]\in \cM_{g-i, 1}$. This time we choose an even theta-characteristic $\eta_C^+$ on $C$
and an  odd theta-characteristic $\eta_D^-$ on $D$. The following family covers the divisor $B_i$:
$$G_i:=\Bigl\{[C\cup_y \cup E\cup _q D, \eta]: \eta_C=\eta_C^+, \eta_E=\OO_E(1), \eta_D=\eta_D^-, E\cap C=\{y\}, E\cap D=\{q\}\Bigr\}_{y\in C}.$$
Clearly $G_i\cdot \alpha_i=0$, $G_i\cdot \beta_i=2-2i$ and $G_i\cdot \lambda=G_i\cdot \alpha_j=G_i\cdot \beta_j=0$ for $j\neq i$.

\subsection{Two elliptic pencils.}
The boundary divisor $\Delta_1\subset \mm_g$ is covered by a standard elliptic pencil $R$ obtained by attaching to a fixed general
pointed curve $[C, y]\in \cM_{g-1, 1}$ a pencil of plane cubic curves $\{E_{\lambda}=f^{-1}(\lambda)\}_{\lambda\in \PP^1}$ where $f:\mathrm{Bl}_9(\PP^2)\rightarrow \PP^1$. The points of attachment on the elliptic pencil are given by a section $\sigma:\PP^1\rightarrow \mathrm{Bl}_9(\PP^2)$ given by one of the base points of the pencil of cubics. We lift this pencil in two possible ways to the space $\ss_g^-$,
depending on the parity of the theta-characteristic on the varying elliptic tail. We fix an even theta-characteristic $\eta_C^{+}\in
\mbox{Pic}^{g-2}(C)$  and
$E\cong \PP^1$ will again denote an exceptional component. We define
the family
$$F_0:=\Bigl\{[C\cup_{q} E\cup_{\sigma(\lambda)} f^{-1}(\lambda),\
 \ \eta_C=\eta_C^{+},\  \eta_E=\OO_E(1),\ \
\eta_{f^{-1}(\lambda)}=\OO_{f^{-1}(\lambda)}]: \lambda\in
\PP^1\Bigr\}\subset \ss_g^-.$$ Since $F_0\cap B_1=\emptyset$, we find
that $F_0\cdot \alpha_1=\pi_*(F_0)\cdot \delta_1=-1$. Similarly,
$F_0\cdot \lambda=\pi_*(F_0)\cdot \lambda=1$ and obviously $F_0\cdot
\alpha_i=F_0\cdot \beta_i=0$ for $2\leq i\leq [g/2]$. For each of
the $12$ points $\lambda_{\infty}\in \PP^1$ corresponding to
singular fibres of $R$, the associated $\eta_{\lambda_{\infty}}\in
\overline{\mathrm{Pic}}^{g-1}(C\cup E \cup
f^{-1}(\lambda_{\infty}))$ are actual line bundles on $C\cup E\cup
f^{-1}({\lambda_{\infty}})$, that is, we do not have to blow-up the
extra node. Thus we obtain that $F_0\cdot \beta_0=0$ and then $F_0\cdot \alpha_0=\pi_*(F_0)\cdot \delta_0=12$.

A second lift of the elliptic pencil to $\ss_g^-$ is obtained by choosing an odd theta-characteristic
$\eta_C^-\in \mbox{Pic}^{g-2}(C)$ whereas on $E_{\lambda}$ one takes each of the $3$ possible even theta-characteristics, that is,
$$G_0:=\Bigl\{\bigl[C\cup_q E\cup_{\sigma(\lambda)} f^{-1}(\lambda), \ \eta_C=\eta_C^{-},
\  \eta_E=\OO_E(1), \eta_{f^-1(\lambda)}\in \gamma^{-1}
[f^{-1}(\lambda)]\bigr]:\lambda\in \PP^1\Bigr\},$$ where $\gamma:\ss_{1, 1}^+\rightarrow \mm_{1, 1}$ is the projection of degree $3$.
Since
$\pi_*(G_0)=3R\subset \Delta_1$, we obtain that $G_0\cdot \lambda=3$. Obviously
$G_0\cdot \alpha_1=0$, hence $G_0\cdot
\beta_1=\pi_*(G_0)\cdot \delta_1=-3$. The map $\gamma:\ss_{1,
1}^+\rightarrow \mm_{1, 1}$ is simply ramified over the point
corresponding to $j$-invariant $\infty$. Hence, $G_0\cdot
\alpha_0=12$ and $G_0\cdot \beta_0=12$.

\subsection{A covering family in $B_0$} We fix a general pointed spin curve $[C, q, \eta_C^-]\in \cS_{g-1, 1}^-$
and as usual $E\cong \PP^1$ denotes an exceptional component. We construct a family of spin curves $H_0\subset B_0$ with general member
$$\bigl[C\cup_{\{y, q\}} E, \ \eta_{C}=\eta_C^-, \ \eta_{E}=\OO_E(1)\bigr]_{y\in C}\subset \ss_g^-$$
and with special fibre corresponding to $y=q$ being the odd spin curve with support $$C\cup_q E'\cup_{q'} E_2\cup _{\{y_2, q_2\}} E,$$ where
$E'$ and $E_2$ are both smooth rational curves and $y_2, q_2\in E$, $E_2\cap E=\{y_2, q_2\}$, while $E_2\cap E'=\{q'\}$. The stable model of this curve is $C\cup_q \bigl(\frac{E_2}{y_2\sim q_2}\bigr)$, having an elliptic tail of $j$-invariant $\infty$. The underlying line bundle $\eta\in \mbox{Pic}^{g-1}(C\cup E'\cup E_2\cup E)$ satisfies $\eta_{C}=\eta_C^-, \eta_{E'}=\OO_{E'}(1), \eta_{E}=\OO_E(1)$ and, for degree reasons,
$\eta_{E_2}=\OO_{E_2}(-1)$. We have the following relations for the numerical parameters of $H_0$:
$$H_0\cdot \lambda=0, \ H_0\cdot \beta_0=1-g, \ H_0\cdot \alpha_0=0,\ H_0\cdot \beta_1=1, \ H_0\cdot \alpha_1=0.$$
(The only non-trivial calculation here uses that $H_0\cdot \beta_0=\pi_*(H_0)\cdot \delta_0/2=1-g$, cf. \cite{HM}).

\section{Theta pencils on K3 surfaces.}
In this section we prove Theorem \ref{uniruled}. As usual, we denote by $\F_g$ the moduli space of polarized $K3$ surfaces $[X, H]$, where $X$ is a $K3$ surface and $H\in \mbox{Pic}(X)$ is a (primitive) polarization of degree $H^2=2g-2$, see \cite{M4}. For an integer $0\leq \delta\leq g$, we introduce the universal \emph{Severi variety} consisting of pairs
$$\mathcal{V}_{g, \delta}:=\Bigl\{\bigl([X, H], C\bigr):
[X, H]\in \F_g \mbox{ and } C\in |\OO_X(H)| \mbox{ is an integral } \delta-\mbox{nodal curve}\Bigr\}.$$
If $\sigma:\mathcal{V}_{g, \delta}\rightarrow \F_g$ is the obvious projection, we set $V_{g, \delta}(|H|):=\sigma^{-1}([X, H])$. It is known that every irreducible component of $\mathcal{V}_{g, \delta}$ has dimension $19+g-\delta$ and maps dominantly onto $\F_g$. It is conjectured that $\mathcal{V}_{g, \delta}$ is irreducible. This is established in \cite{CD} in the range $g\leq 9$ and $g=11$.

For a point $[X, H]\in \F_g$, we consider a pencil of curves
$ P\subset |H|$, and denote by $Z$ the base locus of $P$. We assume that a general member $C \in P$ is a nodal integral curve. It follows that $C - Z$ is smooth and that
$
S :=\mathrm{sing}(C)
$
is a, possibly empty, subset of $Z$. Let
$
\epsilon: X':=\mathrm{Bl}_S(X) \to X
$
be the blow-up of $X$ along the locus $S$ of nodes,  and denote by $E$ the exceptional divisor of $\epsilon$. Let
$$
P' \subset |\epsilon^*H \otimes \OO_{X'}(-2E)|
$$
be the strict transform of $P$ by $\epsilon$, and  $Z'$ its base locus. Since a general member $C \in P$ is nodal precisely along $S$, a general curve $C' \in P'$ is smooth. We view $h':=Z'+E\cdot C'$ as a divisor on the smooth curve $C'$. By the adjunction formula, $h' \in |\omega_{C'}|$.
 \begin{definition}\label{thpe} We say that $P$ is a \emph{theta pencil}, if $h'$ has even multiplicity at each of its points, that is, $\mathcal O_{C'}(\frac 12h')$ is an odd theta-characteristic for every smooth curve $C' \in P'$.
 \end{definition} \par \noindent
The definition  implies that the intersection multiplicity of two curves in $P$ is even at each point $p \in \mathrm{supp}(Z)$. For every pair $[X,H]\in \F_g$ we have that:
\begin{proposition} Every smooth curve $C \in |H|$ belongs to a theta pencil.
\end{proposition}
\begin{proof} Let $\eta$ be an odd theta-characteristic with $h^0(C, \eta)=1$ and write $\eta=\OO_C(d)$, with $d\in C_{g-1}$. Then $\mathbf PH^0\bigl(X, \mathcal I_{d/X}(H)\bigr)$ is a theta pencil. \end{proof}

We can reverse the construction of a theta pencil, starting instead with the normalization of a nodal section of a $K3$ surface. Suppose $$t:=[C', x_1, y_1, \ldots, x_{\delta}, y_{\delta},\  \eta]\in \cM_{g-\delta, 2\delta}\times_{\cM_{g-\delta}} \cS_{g-\delta}^-$$ is a $2\delta$-pointed curve $C'$ together with an isolated odd theta-characteristic $\eta$, such that:

\begin{enumerate}
\item $h^0\Bigl(C', \eta\bigl(-\sum_{i=1}^\delta (x_i+y_i)\bigr)\Bigr)\geq 1$; we write $\eta=\OO_{C'}\bigl(\sum_{i=1}^{\delta}(x_i+y_i)+d\bigr)$, where $d \in C'_{g-3\delta-1}$ is the residual divisor.
\item There exists a polarized $K3$ surface $[X, H]\in \F_g$ and a map $f:C'\rightarrow X$, such that $f(x_i)=f(y_i)=p_i$ for all $i=1, \ldots, \delta$ , \ $f_*(C')\in |H|$, and moreover $f:C'\rightarrow C$ is the normalization map of the $\delta$-nodal curve $C:=f(C')$.
\end{enumerate}

\noindent If $\epsilon:X'\rightarrow X$ is the blow-up of $X$ at the points $p_1, \ldots, p_{\delta}$ and $E\subset X'$ denotes the exceptional divisor, we may view $C'\subset X'$, as a smooth curve in the linear system  $|\epsilon^*H\otimes \OO_{X'}(-2E)|$. Note that
$$\OO_{C'}(C')=K_{C'}\Bigl(-\sum_{i=1}^{\delta}(x_i+y_i)\Bigr)=\eta\otimes \OO_{C'}(d).$$ We  pass to cohomology in the following short exact sequences
$$0\longrightarrow \OO_{X'}\longrightarrow \I_{d/X'}(C')\longrightarrow \OO_{C'}(C')(-d)\longrightarrow 0,$$
and
$$0\longrightarrow \OO_{X'}\longrightarrow \I_{2d+\sum_{i=1}^{\delta} (x_i+y_i)/X'}(C')\longrightarrow \OO_{C'}\longrightarrow  0$$
respectively, in order to obtain that
$$\bigl|\mathcal{I}_{d/X'}(C')\bigr|=\bigl|\mathcal{I}_{2d/X'}(C')\bigr|=\bigl|\mathcal{I}_{2d+\sum_{i=1}^{\delta}(x_i+y_i)/X'}(C')\bigr|=\PP^1$$
is a theta pencil of $\delta$-nodal curves on $X$. The link between this description of a theta pencil and the one provided by Definition \ref{thpe} is given by the relation
$h'=2E\cdot C'+2d$.

If $\K_{g-\delta, \delta}^-\subset \cM_{g-\delta, 2\delta}\times_{\cM_{g-\delta}} \cS_{g-\delta}^-$ is the locus of elements $[C, (x_i, y_i)_{i=1, \ldots, \delta},\  \eta]$ satisfying conditions (i) and (ii), the previous discussion proves the following:
\begin{proposition}\label{kgd}
Every irreducible component of $\K_{g-\delta, \delta}^-$ is uniruled.
\end{proposition}
This implies the following consequence of Proposition \ref{dominant} to be established in the next section:
\begin{theorem}\label{uni9}
We set $g\leq 9$ and $0\leq \delta\leq (g+1)/3$. Then the variety $\K_{g-\delta, \delta}^{-}$ is non-empty, uniruled and dominates the spin moduli space
$\cS_{g-\delta}^-$.
\end{theorem}
\begin{definition} We say that a theta pencil $P$ is $\delta$-\emph{nodal} if its general member is a $\delta$-nodal curve, that is, $|S|=\delta$. We say that $P$ is \emph{regular} if the support $\mathrm{supp}(Z)$ of its base locus consists of $g-1$ distinct points. \end{definition}

A $\delta$-nodal theta pencil $P$ on a $K3$ surface $X$, induces a map
$$m': P'\cong \PP^1 \to \ss^-_{g-\delta},$$
obtained by sending a general $C' \in P'$ to the moduli point  $\bigl[C', \OO_{C'}\bigl(\frac{1}{2}h'\bigr)\bigr]\in \ss_{g-\delta}^-$.
We note in passing that a theta pencil also induces a map
$
m: P' \to \overline {\mathcal S}^-_g
$
defined as follows. Consider the pencil $E + P'$ having fixed component $E$. The general member is a quasi-stable curve $D \in (E + P')$ of arithmetic genus $g$, with exceptional components $\{E_i\}_{i=1, \ldots, \delta}$ corresponding to the exceptional divisors of the blow-up $\epsilon: X'\rightarrow X$. Then
 $$
m(C):=\Bigl[ C\cup \bigl(\cup_{i=1}^{\delta} E_i\bigr), \ \eta_{E_i}=\OO_{E_i}(1), \ \eta_{C'}=\OO_{C'}\bigl(\frac{1}{2} h'\bigr)\Bigr]\in \ss_g^-.
 $$ These pencils will be used extensively in the proof of Theorem \ref{unirational}.


Assume that $[X,H]\in \F_g$ is a general point, in particular $\mbox{Pic}(X)=\mathbb Z\cdot H$. Then every smooth curve $C \in |H|$ is Brill-Noether general, see \cite{La}, which implies that  $h^0(C, \eta) = 1$, for
every odd theta-characteristic $\eta$ on $C$.
Theta pencils with smooth general member define a locally closed subset in the Grassmannian $G(2, H^0(S, \OO_S(H))$ of lines in $|H|$. Let
$
\Theta^-(X,H)
$
be its Zariski closure in $G(2, H^0(S, \OO_S(H))$.
\begin{proposition} $\Theta^-(X,H)$ is pure of dimension $g-1$. \end{proposition}
\begin{proof} Let $f: P^-(X,H) \to |H|$ be the projection map from the projectivized universal bundle over $\Theta^{-}(X, H)$, and $V_{g, 0}(|H|) \subset |H|$ be the open locus of smooth curves. Under our assumptions $f$ has finite fibres over $V_{g, 0}(|H|)$.
Thus $P^-(X,H)$ has pure dimension $g$, and  $\Theta^-(X,H)$ has pure dimension $g-1$. \end{proof}

For a general (thus necessarily regular) theta pencil $P\in \Theta^-(X,H)$, we study in more detail the map $m:P' \to \overline {\mathcal S}^-_g$. 
Let
$
\Delta(X,H) \subset |H|
$
be the discriminant locus. Since $[X,H]\in \F_g$ is general, $\Delta(X,H)$ is an integral hypersurface parameterizing the singular elements of $|H|$. It is well-known that  $\mbox{deg } \Delta(X,H) = 6g + 18$.
\begin{proposition}\label{singthet}
Let $P \in \Theta^-(X,H)$ be a general theta pencil with base locus $Z$. Then every singular curve $C \in P$ is nodal. Furthermore,
$$
P \cdot \Delta(X,H) = 2(a_1 + \dots + a_{g-1}) + b_1 + \dots + b_{4g+20},
$$
where $a_i$ is the parameter point of a curve $A_i\in P$ having a point of $Z$ as its only singularity, and $b_j$ is the parameter point of a curve $B_j\in P$ such that $\mathrm{sing}(B_j) \subset X-Z$. Accordingly, $$P\cdot \alpha_0=4g+20 \ \mbox{ and } P\cdot \beta_0=g-1.$$
 \end{proposition}
\begin{proof}  We set $\mbox{supp}(Z)=\{p_1, \ldots, p_{g-1}\}$. Since $P$ is regular, for $i=1, \ldots, g-1$, there exists a unique curve $A_i \in P$
singular at $p_i$. Moreover, for degree reasons, $p_i$ is the unique double point of $A_i$. Each
pencil $T \subset |H|$ having $p_i$ in its base locus is a tangent line to $\Delta(X,H)$ at $A_i$. Hence the intersection multiplicity  $\bigl(P\cdot \Delta(X,H)\bigr)_{A_i}$ is at least $2$. It follows that the assertion to prove is open on any family of pairs $(P,[X,H])$ such that $P \in \Theta^-(X,H)$. Since $\F_g$ is irreducible, it suffices to produce
one polarized $K3$ surface $(X, H)$ satisfying this condition.

\vskip 4pt

For this purpose, we use \emph{hyperelliptic} polarized $K3$ surfaces $(X,H)$. Consider a rational normal scroll $\mathbb{F}:=\mathbb{F}_a \subset \mathbf P^g$, where $a\in \{0, 1\}$ and $g=2n+1-a$. A general section $R \in |O_{\mathbb F}(1)|$ is a rational normal
curve of degree $g-1$. From the exact sequence
$$
0 \to \mathcal O_{\mathbb F}(-2K_{\mathbb F}-R)  \to \mathcal O_{\mathbb F}(-2K_{\mathbb F}) \to  \mathcal O_R(-2K_{\mathbb F}) \to 0,
$$
one finds that there exist a smooth curve $B \in |-2K_{\mathbb F}|$ and distinct points $o_1, \ldots, o_{g-1}\in B$ such that
the pencil $Q \subset |\OO_{\mathbb F}(R)|$ of hyperplane sections through $o_1, \dots, o_{g-1}$ cuts out a pencil with simple ramification on $B$.

Let $\rho: X \to \mathbb F$ be the double covering of $\mathbb F$ branched along $B$. Then $X$ is a K3 surface and $|H|:=|\OO_X(\rho^*R)|$ is a hyperelliptic linear system on $X$ of genus $g$.  Then $\rho^*(Q)$ is a regular theta pencil on $X$ with the required properties. \end{proof}

Since theta pencils cover $\ss_g^{-}$ when $g\leq 11$ and $g\neq 10$, the following consequence of Proposition \ref{singthet} is very suggestive concerning the variation of $\kappa(\ss_g^-)$ as $g$ increases, in particular, in highlighting the significance of the case $g=12$.

\begin{corollary}\label{Kint}
With the same notation as above, we have that $P\cdot K_{\ss_g^-}=2g-24$. In particular general theta pencils of genus $g<12$ are $K_{\ss_g^-}$-negative.
\end{corollary}
\begin{proof} Use that $(P\cdot \lambda)_{\ss_g^-}=(\pi_*(P)\cdot \lambda)_{\mm_g}=g+1, \ P\cdot \alpha_0=4g+20$ and $P\cdot \beta_0=g-1$.
\end{proof}
\begin{proposition} The locally closed set of nodal theta pencils in $\Theta^-(X,H)$ is non empty. If $P$ is a general nodal theta pencil, then  a general
curve $C \in P$ has one node as its only singularity.
\end{proposition}
\begin{proof} We keep the notation from the previous proof and construct a smooth curve $B \in |-2K_{\mathbb F}|$ and choose general points $o, o_1, \ldots, o_{g-3}\in B$, such that the pencil $Q \subset |\OO_{\mathbb F}(R)|$ of the hyperplane sections through $o_1 + \dots+ o_{g-3} + 2o$ cuts out a pencil with simple ramification on $B$. Then $\rho^*(Q)$ is a nodal theta pencil with the required properties.
 \end{proof}


\begin{theorem} $\ss^-_g$ is uniruled for $g \leq 11$. \end{theorem}

\begin{proof} By [M1-4], a general curve $[C]\in \mm_g$ is embedded in a K3 surface $X$ precisely when $g \leq 9$ or $g = 11$. By Proposition \ref{singthet},
$C$ belongs to a theta pencil $P \subset |\OO_X(C)|$ (which moreover, is $K_{\ss_g^-}$-negative). Thus the statement follows for $g \leq 9$ and $g = 11$.
\noindent
To settle the case of $\ss_{10}^-$, we show that $\K_{10, 1}^-$ is non-empty and irreducible. Indeed, then by Proposition \ref{kgd} it follows that $\K_{10, 1}^-$ is uniruled, and since the projection map $\K_{10, 1}^-\rightarrow \cS_{10}^-$ is finite, $\K_{10, 1}^-$ dominates $\cS_{10}^-$. This implies that $\ss_{10}^-$ is uniruled.

The variety $\K_{10, 1}^-$ is an open subvariety of the irreducible locus
$$\cU:=\Bigl\{\bigl([C, x, y], \eta \bigr)\in \cM_{10, 2}\times_{\cM_{10}} \cS_{10}^-: h^0(C, \eta\otimes \OO_C(-x-y))\geq 1\Bigr\},$$ hence it is irreducible as well. To establish its non-emptiness, it suffices to produce an example of an element $\bigl([C, x, y],  \eta]\bigr)\in \cU$, such that the curve $C_{xy}$ can be embedded in a $K3$ surface. We specialize to the case when $C$ is hyperelliptic and $x, y\in C$ are distinct Weierstrass points, in which case one can choose $\eta=\OO_C(x+y+w_1+\cdots+ w_7)$, where $w_i$ are distinct Weierstrass points in $C-\{x, y\}$. Again we let $\rho: X\rightarrow \mathbb F\subset \PP^{11}$ be a hyperelliptic $K3$ surface branched along $B\in |-2K_{\mathbb F}|$, with polarization $H:=\rho^* \OO_{\mathbb F}(1)$, so that $[X, H]\in \F_{11}$. We set $C:=\rho^*(R)$, where $R\in |\OO_{\mathbb F}(1)|$ is a rational normal curve of degree $10$. We need to ensure that $C$ is $1$-nodal, with its node $p\in C$ such that if $f:C'\rightarrow C$ denotes the normalization map, then both points in $f^{-1}(p)$ are Weierstrass points. This is satisfied once we choose $R$ in such a way that $B\cdot R\geq  2\rho(p)$.
\end{proof}

\section{Unirationality  of $\mathcal \ss^-_g$  for  $g\leq 8$}

To prove the claimed unirationality results, we use that a general curve $[C]\in \mm_g$ has a sextic plane model when $g\leq 6$, or is a linear section of a Mukai variety, when $7\leq g\leq 9$. We start with the easy case of small genus, before moving on to the more substantial study of Mukai models.

\begin{theorem}\label{genatmost6}
$\ss_g^-$ is unirational for $g\leq 6$.
\end{theorem}
%
\begin{proof} We fix $3\leq g\leq 6$ and a general odd spin curve $[C, \eta]\in \cS_g^-$. Write $\eta=\OO_C(d)$, where $d\in C_{g-1}$, then choose a general linear system $A\in G^2_6(C)$. The induced morphism $\phi_A\colon C\rightarrow \Gamma\subset \PP^2$ realizes $C$ as a sextic with $\delta=10-g$ nodes.  By choosing $[C, \eta]$ and $A$ generically, we may assume that $\mbox{supp}(d)$ consists of $g-1$ points and is disjoint from $\phi_A^{-1}(\mbox{sing}(\Gamma))$. Accordingly, we identify $d$ with its image  $\phi_A(d)$ on $\Gamma$. By adjunction
$$\OO_C(2d)=\omega_C=\OO_C(3)\Bigl(-\phi_A^{-1}(\mbox{sing}(\Gamma))\Bigr),$$ therefore the unique plane cubic $E\in |\OO_{\PP^2}(3)|$ passing through the $10-g$ nodes of $\Gamma$ as well as through the $g-1$ points of $\mbox{supp}(d)$ is actually tangent to $\Gamma$ along $\mbox{supp}(d)$.

\vskip 3pt

We denote by $\cU$ the parameter space of elements $t=\bigl(E, o, \tau, x_1, \ldots, x_{10-g}, d\bigr)$, where $[E, o]\in \cM_{1,1}$, $\tau\in \mbox{Pic}^0(E)[2]$ is a $2$-torsion point, $x_1, \ldots, x_{10-g}\in E$ and an effective divisor  $d\in \bigl|\OO_E(9o)\otimes \tau(-x_1-\cdots-x_{10-g})\bigr|$. 
Over $\cU$ there lies a projective bundle
$\P$ whose fibre over $t$ is the linear system of plane sextics $\Gamma$ which are singular at the points $x_1, \ldots, x_{10-g}$ and totally tangent to $E$ along the divisor $d$. Then $\P$ is a rational variety, and by the previous remark, it dominates $\ss_g^-$. Thus $\ss_g^-$ is unirational\footnote{The last lines were added in September 2026. We thank Vlad Robu for pointing this out to us.}.
\end{proof}

We assume now that $7 \leq g \leq 10$ and denote by $
V_g \subset \mathbf P^{N_g}
$ the rational homogeneous space defined as follows, see  \cite{M1}, \cite{M2}, \cite{M3}:
\bigskip  \par \noindent \it
- $V_{10}$: the $5$-dimensional variety $G_2/P\subset \PP^{13}$ corresponding to the Lie group $G_2$, \par \noindent
- $V_9$:  the Pl\"ucker embedding of the symplectic Grassmannian $SG(3,6)\subset \PP^{13}$, \par \noindent
- $V_8$:  the Pl\"ucker embedding of the Grassmannian $G(2,6)\subset \PP^{14}$, \par \noindent
- $V_7$:  the Pl\"ucker embedding of the orthogonal Grassmannian $OG(5,10)\subset \PP^{15}$. \bigskip  \par \noindent   \rm
Note that $N_g=g+\mbox{dim}(V_g)-2$. Inside the Hilbert scheme $\mbox{Hilb}(V_g)$ of curvilinear sections of $V_g$, we consider the open set \
$
\cU_g
$
classifying curves $C\subset V_g$ such that \medskip
\begin{itemize} \it
\item $C$ is a nodal integral section of $V_g$ by a linear space of dimension $g-1$,
\item the residue map $\rho: H^0(C, \omega_C) \to H^0(C, \omega_C \otimes \mathcal O_{\mathrm{sing}(C)})$ is surjective.
\end{itemize} \medskip \par \noindent
 A general point $[C\hookrightarrow \PP^{g-1}]\in \cU_g$ is a smooth, canonical curve of genus $g$. Mukai's results \cite{M1}, \cite{M2}, \cite{M3} imply that $C$ has general moduli if $g \leq 9$. For each $0 \leq \delta\leq g-1$, we define the locally closed sets of $\delta$-nodal curvilinear sections of $V_g$
$$
\cU_{g, \delta} := \lbrace [C\hookrightarrow \PP^{g-1}] \in \cU_g : |\mathrm{sing}(C)| = \delta \rbrace.
$$
\begin{proposition}\label{deltanodal} For $g\leq 9$, the variety $\cU_{g, \delta}$ is smooth of pure codimension $\delta$ in $\cU_g$.
\end{proposition}
\begin{proof} A general $2$-dimensional linear section of $V_g$ is a polarized K3 surface $[X, H]\in \F_g$ with general moduli. It is known \cite{Ta} that $\delta$-nodal hyperplane sections of $S$ form a pure $(g-\delta)$-dimensional family $V_{g, \delta}(|H|)\subset |H|$. Thus,  $\cU_{g, \delta}\neq \emptyset$ and  $\mbox{codim}(\cU_{g, \delta}, \cU_g)\leq \delta$.  We fix a curve $[C] \in \cU_{g,\delta}$, then consider the normal bundle $N_C$ of $C$ in $V_g$ and the map $r\colon H^0(C, N_C) \to \mathcal O_{\mathrm{sing}(C)}$ induced by the exact sequence
\begin{equation}\label{schlessinger}
0 \to T_C \to T_{V_g}\otimes \OO_C \to N_C \stackrel{r}\rightarrow  T_C^1 \to 0,
\end{equation}
where $T_C^1=\mathcal O_{\mathrm{sing}(C)}$ is the Lichtenbaum-Schlessinger sheaf of $C$
classifying  the deformations of $\mbox{sing}(C)$.
 Using the identification $T_{[C]}(\cU_g)=H^0(C, N_C)$, it is known that $\mbox{Ker}(r)$ is isomorphic to  $T_{[C]}(\cU_{g,\delta})$, see e.g. \cite{HH}. Furthermore, $N_C \cong \omega_C^{\oplus (N_g-g+1)}$ and $r = \rho^{\oplus (N_g-g+1)}$, where $\rho: H^0(C, \omega_C) \to H^0 (C, \mathcal O_{\mathrm{sing}(C)})$ is the map given by the residues at the nodes. Since $\rho$ is surjective, $\mbox{Ker}(r)$ has codimension $\delta$ inside $T_{[C]}(\cU_g)$ and the statement follows. \end{proof}

 \par \noindent
The automorphism group $\mbox{Aut}(V_g)$ acts in the natural way on $\mbox{Hilb}(V_g)$. The locus of singular curvilinear sections $[C]\in \cU_g$ is an $\mbox{Aut}(V_g)$-invariant divisor which misses a general point of $\cU_g$, therefore $\cU_g^{\mathrm{ss}}:=\cU_g\cap \mbox{Hilb}(V_g)^{\mathrm{ss}}\neq \emptyset$. Since $\rho(V_g)=1$, the notion of stability is independent of the polarization. The (quasi-projective) GIT-quotient
$$
\mathfrak{M}_g := \cU_g^{\mathrm{ss}} \dblq \mbox{Aut}(V_g)
$$ is said to be the \emph{Mukai model} of $\mm_g$.
We have the following commutative diagram
$$
\begin{CD}
{\cU_g^{\mathrm{ss}}} @>{}>> {\cU_g} \\
@V{u_g}VV @V{m_g}VV \\
{\mathfrak{M}_g} @>{\phi_g}>> {\overline {\mathcal M}_g} \\
\end{CD}
$$
where $u_g\colon \cU_g^{\mathrm{ss}} \rightarrow \mathfrak{M}_g$ is the quotient map and  $m_g\colon \cU_g \to \overline {\mathcal M}_g$ is the moduli map. The general fibre of $m_g$ is an $\mbox{Aut}(V_g)$-orbit. Summarizing results from \cite{M1}, \cite{M2}, \cite{M3}, we state the following:
\begin{theorem} For $7\leq g \leq 9$, the map $\phi_g: \mathfrak{M}_g \dashrightarrow \mm_g$ is a birational isomorphism. The inverse map $\phi_g^{-1}$ contracts the (unique) Brill-Noether divisor $\mm_{g, d}^r\subset \mm_g$ of curves with a $\mathfrak g^r_d$ when $\rho(g, r, d)=-1$, as well as the boundary divisors $\Delta_i$ with $1\leq i\leq [g/2]$.
 \end{theorem}
Next, let
$
\Delta_{g}^{\delta} \subset \Delta_0\subset \mm_g
$
be the locus of integral stable curves of arithmetic genus $g$ with $\delta$ nodes. Then $\Delta_g^{\delta}$ is irreducible of codimension $\delta$ in $\mm_g$.
\begin{lemma}\label{dominant} Set $7\leq g \leq 9$ and let $D$ be any irreducible component of $\cU_{g, \delta}$. Then the restriction morphism $m_{g | D}: D \to \Delta_{g}^{\delta}$ is dominant. In particular, a general $\delta$-nodal curve $[C]\in \Delta_g^{\delta}$ lies on a smooth $K3$ surface.
 \end{lemma}
\begin{proof}  Since $\cU_{g,\delta}$ is smooth, $D$ is a connected component of $\cU_{g,\delta}$, that is,  for $[C] \in D$, the tangent spaces to $D$ and to $\cU_{g,\delta}$ coincide. We consider again the sequence (\ref{schlessinger}):
$$
0 \to T_C \to T_{V_g}\otimes \OO_C \to N'_C  \to 0,
$$
where $N'_C:=\mbox{Im }\{T_{V_g}\otimes \OO_C \to N_C\}$ is the \emph{equisingular sheaf} of $C$. We have that $H^0(C, N_C') = \mbox{Ker}(r)$. As  remarked in the proof of Proposition \ref{deltanodal}, $H^0(C, N_C')$ is the tangent space $T_{[C]}(\cU_{g,\delta})$ and its codimension in $H^0(C, N_C)$ equals $\delta$. Consider the coboundary map
$\partial: H^0(C, N_C') \to H^1(C, T_C)$. Since $H^1(C, T_C)$ classifies topologically trivial deformations of the nodal curve $C$, the image $\mbox{Im}(\partial)$ is isomorphic to the image of the tangent map $dm_{g | \cU_{g, \delta}}$ at $[C]$. On the other hand $H^0(C, T_{V_g}\otimes \OO_C)$ is the  tangent space to
the orbit of $C$ under the action of $\mbox{Aut}(V_g)$. This is reduced and the stabilizer of $C$, being a subgroup of $\mbox{Aut}(C)$,  is finite,  hence we obtain:
$$
\mbox{dim } \ \mbox{Im}(\partial)  =  h^0(C, N_C) - \delta - \mbox{dim } \mbox{Aut}(V_g)  = 3g - 3 - \delta.
$$
Since $\Delta_g^{\delta}$ has codimension $\delta$ in $\overline {\mathcal M}_g$, it follows that $m_{g | D}$ is dominant. \end{proof} \par \noindent

\begin{proposition} Fix $0\leq \delta\leq g-1$ and $D$ an irreducible component of $\cU_{g, \delta}$. Then $D^{\mathrm{ss}}\neq \emptyset$.
\end{proposition}
\begin{proof}
It suffices to construct an $\mbox{Aut}(V_g)$-invariant divisor which does not contain $D$. We carry out the construction when $g=8$, the remaining cases being largely similar.

We fix a complex vector space $V\cong \mathbb C^6$, and then $V_8:=G(2, V)\subset \PP(\wedge^2 V)$ and $\cU_8\subset G(8, \wedge^2 V)$. For a projective $7$-plane $\Lambda \in G(8, \wedge ^2 V)$, we denote the set of containing hyperplanes
$F_{\Lambda}:=\{H\in \PP(\wedge^2 V)^{\vee}: H\supset \Lambda\}$, and define the $\mbox{Aut}(V_8)$-invariant divisor
$$Z:=\{\Lambda \in \cU_8: F_{\Lambda}\cap G(2, V^{\vee})\subset \PP(\wedge^2 V)^{\vee} \mbox{ is not a transverse intersection}\}.$$
We claim that $D\nsubseteq Z$. Indeed, let us fix a general point $[C\hookrightarrow \Lambda]\in D$, where $\Lambda=\langle C\rangle$, corresponding to a general curve $[C]\in \Delta_g^{\delta}$. In particular, we may assume that $C$ lies outside the closure in $\mm_g$ of curves violating the Petri theorem. Thus $C$ possesses no generalized $\mathfrak g^2_7$'s, that is, $\overline{W}^2_7(C)=\emptyset$, whereas $\overline{W}^1_5(C)\subset \mbox{Pic}(C)$ consists of locally free pencils
 satisfying the Petri condition. We recall from \cite{M2} the construction of $\phi_g^{-1}[C]$, which generalizes to irreducible Petri general nodal curves: There exists a unique rank two vector bundle $E$ on $C$ with $\mbox{det}(E)=\omega_C$ and $h^0(C, E)=6$. This appears as an extension
 $$0\rightarrow A\rightarrow E\rightarrow \omega_C\otimes A^{\vee}\rightarrow 0,$$ for every $A\in \overline{W}^1_5(C)$. Then one sets $\phi_g^{-1}([C]):=[C\hookrightarrow G(2, H^0(C, E)^{\vee})].$ Moreover, $$F_{\Lambda}=\PP\bigl(\mbox{Ker}\{\wedge^2 H^0(C, E)\rightarrow H^0(C, \omega_C)\}\bigr).$$ In particular, the intersection  $F_{\Lambda}\cap G(2, H^0(C, E))$ corresponds to the pencils $A\in \overline{W}_5^1(C)$. Since $C$ is Petri general, $\overline{W}_5^1(C)$ is a smooth scheme, thus $[C\hookrightarrow \Lambda]\notin Z$.
 \end{proof}
 We consider the quotient $\mathfrak{M}_{g, \delta} := \cU_{g, \delta}^{\mathrm{ss}} \dblq \mbox{Aut}(V_g)$ and the induced map
$$
\phi_{g,\delta}: \mathfrak{M}_{g,\delta} \rightarrow \Delta_{g}^{\delta}.
$$

\begin{theorem} The variety  $\mathfrak{M}_{g,\delta}$ is irreducible and $\phi_{g,\delta}$ is a birational isomorphism.
\end{theorem}
\begin{proof} By Lemma \ref{dominant}, any irreducible component $Y$ of $\mathfrak{M}_{g,\delta}$ dominates $\Delta_{g}^\delta$.  On the other hand, $\phi_g: \mathfrak{M}_g \rightarrow \overline{\mathcal M}_g$ is a birational morphism and  $\phi_{g,\delta} = \phi_{g | \mathfrak{M}_{g, \delta}}$. Since $\overline {\mathcal M}_g$ is normal, each fibre of $\phi_g$ is connected, thus $\mathfrak{M}_{g,\delta}$ is irreducible and $\mbox{deg}(\phi_{g,\delta}) = 1$.
\end{proof} \par \noindent

\vskip 3pt

We lift our construction to  the space of odd spin curves. Keeping $7\leq g \leq 9$, we consider the Hilbert scheme
$
\mbox{Hilb}_{2g-2}(V_g)
$
of $0$-dimensional subschemes of $V_g$ having length $2g-2$.
\begin{definition} Let $\mathfrak Z_{g-1}\subset \mathrm{Hilb}_{2g-2}(V_g)$ be the parameter space of those $0$-dimensional schemes $Z\subset V_g$ such that:
\begin{itemize}
\item[(1)] $Z$ is a hyperplane section of a smooth curve section $[C] \in \cU_g$,
\item[(2)] $Z$ has multiplicity two at each point of its support,
\item[(3)] $\mathrm{supp}(Z)$ consists of $g-1$ linearly independent points.
\end{itemize} \par \noindent
\end{definition} \par \noindent
The space $\mathfrak{Z}_{g-1}$  classifies \emph{clusters} of length $2g-2$ on $V_g$. The cycle associated under the Hilbert-Chow morphism to a general point of $\mathfrak{Z}_{g-1}$ corresponds to a $0$-cycle of the form $2p_1+\cdots+2p_{g-1}\in \mbox{Sym}^{2g-2}(V_g)$ satisfying the condition  $$\mbox{dim } \langle p_1, \ldots, p_{g-1}\rangle\cap \mathbb T_{p_i}(V_g)\geq 1, \ \mbox{ for } \ i=1, \ldots, g-1.$$
Clearly $\mbox{dim}(\mathfrak{Z}_{g-1})=\mbox{dim } \textbf{G}(g-1, N_g)-(N_g-g+1)=(g-1)(N_g-g+1)$. We consider the incidence correspondence between clusters and curvilinear sections
of $V_g$
$$
\cU^-_g :=\{(C,Z) \in \cU_g \times \mathfrak{Z}_{g-1}: Z \subset C\}.
$$
The first projection map
$\pi_1:\cU^-_g \to \cU_g$
is finite of degree $2^{g-1}(2^g-1)$; its fibre at a general point $[C]\in \cU_g$ is in bijective correspondence with the set of odd theta-characteristics of $C$.  In particular, both $\cU_g^{-}$ and $\mathfrak Z_{g-1}$ are irreducible varieties. The spin moduli map
$$
m^-_g:\cU^-_g \dashrightarrow \overline {\mathcal S}_g^-
$$
is defined by $m_g^-(C,Z):=[C, \OO_C(Z/2)]$, for each point $(C, Z)\in \cU_g^-$ corresponding to a smooth curve $C$. Later we shall extend the rational map $m_g^-$ to a regular map over $\cU_g^-$. It is clear that $m^-_g$ induces a  map
$
\phi^-_g: Q^-_g \dashrightarrow \ss_g^-
$
from the quotient $$Q^-_g := \pi_1^{-1}(\cU_g^{\mathrm{ss}}) \dblq \mbox{Aut}(V_g).$$
We may think of $Q_g^-$ as being the \emph{Mukai model} of $\ss_g^-$. If $\pi^-: Q_g^- \rightarrow \mathfrak{M}_g$ is the map induced by $\pi$ at the level of Mukai models, we have a commutative diagram:
$$
\begin{CD}
{Q^-_g} @>{\phi^-_g}>> {\overline{\mathcal S}^-_g} \\
@V{\pi^-}VV @V{\pi}VV \\
{\mathfrak{M}_g} @>{\phi_g}>> {\overline {\mathcal M}_g} \\
\end{CD}
$$
\begin{proposition} The spin Mukai model $Q_g^-$ is irreducible and $\phi^-_g: Q_g^- \rightarrow \overline{\mathcal S}^-_g$ is a birational isomorphism. \end{proposition}

\vskip 3pt
One extends the rational map $m^-_g$ (therefore $\phi_g^-$ as well) to a regular morphism over the locus of points with nodal underlying curve section of $V_g$ as follows.  Let  $(C,Z) \in \cU^-_g$ be an arbitrary point and set $\mbox{supp}(Z) := \{p_1, \dots, p_{g-1}\}$. Assume that
$
\mathrm{sing}(C) \cap \mathrm{supp}(Z) = \lbrace p_1, \dots, p_{\delta} \rbrace
$, where $\delta\leq g-1$.
Consider the partial normalization
$\nu: N \to C$ at the points $p_1, \ldots, p_{\delta}$. In particular, there exists an effective Cartier divisor $e$ on $C$ of degree $g-\delta-1$,  such that  $2e = Z \cap (C - \mathrm{sing}(C))$, and set
$\epsilon := \mathcal O_N(\nu^*e).
$
Then $m_g^-(C,Z)$ is the spin curve $[X, \eta]\in \ss_g^-$ defined as follows:
\begin{definition}\label{spinde} We describe the following stable spin curve:
\begin{itemize}
\item[(1)] $X := N \cup E_1 \cup \dots \cup E_{\delta}$, where $E_i = \mathbf P^1$ for $i =1, \ldots, \delta$.
\item[(2)] $E_i \cap N = \nu^{-1}(p_i)$,
for every node $p_i \in \mathrm{sing}(C)\cap \mathrm{supp}(Z)$. \par \noindent
\item[(3)] $\eta \otimes \mathcal O_N \cong \epsilon$ and $\eta \otimes \mathcal O_{E_i} \cong \mathcal O_{\mathbf P^1}(1)$.
\end{itemize}
\end{definition}  \par \noindent
We note that $N$ is smooth of genus $g -\delta$, precisely when $\mathrm{sing}(C) \subset  \mathrm{supp}(Z)$. In this case $\epsilon\in \mbox{Pic}^{g-1-\delta}(N)$ is a theta characteristic and $h^0(N, \epsilon)=1$. Observe also that there is an isomorphism
$H^0(X,\omega_X\otimes \eta^{\otimes (-2)})\cong H^0(N, \omega_N\otimes \epsilon^{\otimes (-2)})=\mathbb C$, so the spin curve in Definition \ref{spinde} is uniquely determined by specifying $X$ and $\eta$.

\vskip 3pt
For $1\leq \delta\leq g-1$, we refine our incidence correspondence and consider
$$
\cU^-_{g,\delta} :=\Bigl\{(C, Z) \in \cU^-_{g}: \mathrm{sing}(C)\subset \mathrm{supp}(Z), \ \ |\mathrm{sing}(C)|= \delta \Bigr\}.
$$
We denote by $B_{g, \delta}^-$ the closure of $m_g^-(\cU^-_{g,\delta})$ inside $\ss_g^-$; this is the closure in $\ss_g^-$ of the locus of $\delta$-nodal spin curves having $\delta$ exceptional components. Clearly $B_{g, \delta}^-$ is an irreducible component of $\pi^{-1}(\Delta_g^{\delta})$ and it is birationally isomorphic to $\ss_{g-\delta, 2\delta}/\mathbb Z_2^{\delta}$.  We set
$$Q^-_{g,\delta} := \cU_{g, \delta}^- \cap \pi_1^{-1}(\cU_{g}^{\mathrm{ss}}) \dblq \mbox{Aut}(V_g),$$ and let $u^-_g: \cU_{g, \delta}^-\dashrightarrow Q_{g, \delta}^-$ denote the quotient map. Keeping all previous notation, we have a further commutative diagram
$$
\begin{CD}
{\cU^-_{g,\delta}} @>{u^-_g}>> {Q^-_{g,\delta}} @>{\phi^-_{g,\delta}}>> {B^-_{g,\delta}}  \\
@V{}VV @V{\pi^{-}}VV @V{\pi}VV \\
{\cU_{g,\delta}}@>{u_g}>> {\mathfrak{M}_{g,\delta}} @>{\phi_{g,\delta}}>> {\Delta_{g}^{\delta}}  \\
\end{CD}
$$
where $\phi^-_{g,\delta}$ is the morphism induced on $Q^-_{g,\delta}$ by $m_g^-$.
\begin{theorem} We fix  $7 \leq g \leq 9$ and $1 \leq \delta \leq g-1$. Then $\phi^-_{g,\delta}: Q_{g,\delta}^- \rightarrow  B_{g, \delta}^-$ is a birational isomorphism.
\end{theorem}
\begin{proof} It suffices to note that $\phi_{g, \delta}$ is birational, and the vertical arrows of the diagram are finite morphisms of the same degree, namely the number of odd theta-characteristics on a curve of genus $g - \delta$.
\end{proof}
\par \noindent
We construct a projective bundle over $B^-_{g,\delta}$, then show that for certain values $\delta\leq g-1$, the locus  $B^-_{g, \delta}$ itself is unirational, whereas the above mentioned bundle dominates $\mathcal S^-_g$.  Let
$
\mathcal C_{g,\delta} \subset \cU^-_{g,\delta} \times V_g
$
be the universal curve, endowed with its two projection maps
$$
\begin{CD}
{\cU^-_{g, \delta}} @<p<< {\mathcal C_{g,\delta}} @>q>> {V_g} \\
\end{CD}.
$$
We fix a point $(\Gamma, Z)\in \cU_{g, \delta}^-$ and let $\nu: N \to \Gamma$ be the normalization map. Recall that $\mbox{sing}(\Gamma)$ consists of $\delta$ linearly independent points and that  $h^0(N, \mathcal O_N(\nu^*e)) = 1$, where $e$ is the effective divisor on $\Gamma$ characterized by $Z_{|\Gamma_{\mathrm{reg}}}=2e$. Thus the restriction map
$H^0(\Gamma, \omega_{\Gamma}) \to H^0(\omega_{\Gamma} \otimes \mathcal O_Z)$
has $1$-dimensional kernel. In particular the relative cotangent sheaf $\omega_p$ admits a global section $s$ inducing an exact sequence
$$
0 \to \mathcal O_{\mathcal C_{g,\delta}} \to \omega_p \to \mathcal O_W \otimes \omega_p \to 0,
$$
which defines a subscheme $W\subset \cC_{g, \delta}$,  whose fibre at the point $(\Gamma, Z)\in \cU_{g, \delta}^-$ is $Z$ itself.  We set
$$
\mathcal A := p_*\bigl(\mathcal I_{W/\cC_{g, \delta}} \otimes q^* \mathcal O_{V_g}(1)\bigr),
$$
which is a vector bundle on $\cU^-_{g, \delta}$ of rank $N_g - g + 2$. The fibre of $\mathcal A(\Gamma,Z)$ is identified with  $H^0(V_g, \mathcal I_{Z/V_g}(1))$.
One has a natural identification
$$
\mathbf PH^0(\mathcal I_{Z/V_g}(1))^{\vee} = \lbrace \text{\it 1-dimensional linear sections of $V_g$ containing $Z$} \rbrace.
$$
\begin{definition} \it $\mathcal{P}_{g, \delta}$ is the projectivized dual of $\mathcal A$.
\end{definition} \par \noindent
From the definitions and the previous remark it follows:
\begin{proposition} $\mathcal{P}_{g,\delta}$ is the Zariski closure of the incidence correspondence
$$
\mathcal{P}_{g, \delta}^{o} := \Bigl\{ \bigl(C, (\Gamma, Z)\bigr) \in \cU_g \times \cU^-_{g, \delta} \ : Z \subset C \Bigr\}.
$$
\end{proposition} \par \noindent
Consider the projection maps
$$
\begin{CD}
{\cU^-_g} @<{\alpha}<< {\mathcal{P}^o_{g, \delta}} @>{\beta}>> {\cU^-_{g, \delta}}. \\
\end{CD}
$$
We wish to know when is $\alpha$ a dominant map.  For $1 \leq \delta < g \leq 9$, we have the following:

\begin{proposition} The morphism $\alpha$ is dominant if and only if  $\delta \leq N_g + 1 - g=\mathrm{dim}(V_g)-1$.
\end{proposition}
\begin{proof} By definition, the morphism $\beta$ is surjective. Let  $(\Gamma, Z) \in \cU^-_{g, \delta}$ be an arbitrary point, and  set $\mathrm{sing}(\Gamma):=\{p_1, \ldots,  p_{\delta} \rbrace \subset Z$. We define
$\PP_Z$ to be the locus of one-dimensional linear sections of $V_g$ containing $Z$.
Inside $\PP_Z$ we consider the space
$$
\PP_{_{\Gamma, Z}} := \bigl\{\Gamma' \in \PP_Z: \mathrm{sing}(\Gamma') \cap Z \supseteq \mathrm{sing}(\Gamma) \cap Z \bigr\},
$$

\noindent
For $p\in \mathrm{sing}(\Gamma)$, the locus
$H_p := \lbrace \Gamma' \in \PP_Z: p \in \mathrm{sing}(\Gamma') \rbrace
$
 is a hyperplane in $\PP_Z$. Indeed, we identify $\PP_Z$ with the family of linear spaces $L\in \textbf G(g-1, N_g)$ such that $\langle Z\rangle\subset L$. By the definition of the cluster $Z$, it follows that $\mathbb T_p(V_g) \cap \langle Z \rangle$ is a line. For $L\in \PP_Z$, the intersection  $L \cap V_g$ is singular at $p$ if and only if
 $\mathrm{dim } \  L\cap \mathbb T_p(V_g) \geq 2$.  This is obviously a codimension $1$ condition in $\PP_Z$.
Therefore, if for $1\leq i\leq \delta$ we define the hyperplane $H_i := \{L=\langle \Gamma'\rangle \in \PP_Z: \mathrm{dim }\  L\cap \mathbb T_{p_i}(V_g) \geq 2\}$, then
 $$
 \PP_{{\Gamma, Z}} = H_1 \cap \dots \cap H_{\delta}.
 $$ This shows that the general point in $\beta^{-1}(C, Z)$ corresponds to a smooth curve $C\supset Z$. We now fix a general point $(\Gamma, Z)\in \cU_{g, \delta}^-$, corresponding to a general cluster $Z\in \mathfrak{Z}_{g-1}$.

\vskip 5pt

\noindent
\it Claim:  $\PP_{{\Gamma, Z}}$ has codimension $\delta$ in $\PP_Z$; its general element is a nodal curve with $\delta$ nodes.  \rm \par 

\noindent
\emph{Proof of the claim:} Indeed $\PP_Z$ is a general fibre of the projective bundle $\cU_g^-\rightarrow \mathfrak Z_{g-1}$. The claim follows since
$\mbox{codim}(\cU_{g, \delta}^-, \cU_g^-)=\delta$.

 \vskip 4pt
\noindent
The fibre $\alpha^{-1}\bigl((C, Z)\bigr)$ over a general point $(C, Z)\in \cU_g^-$, is the union of $\binom {g-1}{\delta}$ linear spaces
$H_1 \cap \dots \cap H_{\delta} \subset \PP_Z$ as above. By the claim above, when $Z\in \mathfrak{Z}_{g-1}$ is a general cluster, this is a union of linear spaces $\PP_{_{\Gamma, Z}}$ as before, having codimension $\delta$ in
$\PP_Z$. Hence $\alpha^{-1}\bigl((C, Z)\bigr)$ is not empty if and only if $\delta \leq \mathrm{dim} \ \PP_Z$, that is, $\delta\leq N_g-g+1$. \end{proof}
\par \noindent
Let us fix the following notation:
\begin{definition} \it \ \par \noindent (1) $\overline{\mathbb P}_{g,\delta} := \bigl(\P^o_{g, \delta}\bigr)^{\mathrm{ss}} \dblq \mathrm{Aut}(V_g)$. \par \noindent
(2) $\overline{\alpha}: \overline{\mathbb P}_{g, \delta} \to \ss_g^- $ is the morphism induced by $\alpha$ at the level of quotients.
\end{definition} \par \noindent
Note that $\beta: \P_{g, \delta} \to \cU^-_{g,\delta}$ is a projective bundle and $\mbox{Aut}(V_g)$ acts linearly on its fibres, therefore
$\beta$ descends to a projective bundle on $B^-_{g,\delta}$.  Then it follows from the previous remark that $\P_{g, \delta}$ is birationally isomorphic to $\PP^{N_g - g + 1} \times B^-_{g, \delta}$. To finish the proof of the unirationality of $\mathcal S^-_g$, we proceed as follows:
\begin{theorem}\label{assu} Let $7 \leq g \leq 9$ and assume that (i) $B^-_{g, \delta}$ is unirational and (ii)   $\delta \leq N_g - g + 1$.
Then $\mathcal \ss^-_g$ is unirational.
\end{theorem}
\begin{proof} By assumption (ii), the map  $\alpha: \P^o_{g, \delta} \to \cU^-_g$ is dominant, Hence the same is true for the induced morphism $\overline{\alpha}: \overline {\mathbb P}_{g, \delta}  \to \overline {\mathcal S}^-_g$. By (i) and the above remark, $\overline {\mathbb P}_{g,\delta}$ is unirational. Therefore $\ss^-_g$ is unirational as well. \end{proof}

Theorem \ref{assu} has some straightforward applications. The case
$\delta = g-1$
is particularly convenient, since $B^-_{g, g-1}$  is isomorphic to the moduli space of integral curves of geometric genus $1$ with $g-1$ nodes. For $\delta=g-1$, the assumptions of Theorem \ref{assu} hold when $g\leq 8$. In this range, the unirationality of $\mathcal S^-_g$ follows from that of $B_{g,g-1}^-$.
\begin{theorem}\label{uni8}
$B_{g,g-1}^-$ is unirational for $g \leq 10$. \end{theorem}
\begin{proof} Let $I \subset \mathbf P^2 \times (\mathbf P^{2})^{\vee}$ be the natural incidence correspondence consisting of pairs $(x,\ell)$
such that $x$ is a point on the line $\ell$. For $\delta \leq 9$, we define
$$
\Pi_{\delta}:= \bigl\{ (x_1, \ell_1, \dots, x_{\delta}, \ell_{\delta}, E) \in I^{\delta} \times \mathbf PH^0(\PP^2, \OO_{\PP^2}(3)):  \ x_1, \dots, x_{\delta} \in E \bigr\}.
$$
Then there exists a rational map $ f_{\delta}: \Pi_{\delta} \dashrightarrow B_{\delta + 1, \delta}^-$ sending $(x_1,\ell_1, \dots, x_{\delta}, \ell_{\delta}, E)$ to the  moduli point of the $\delta$-nodal, integral  curve $C$ obtained from  the elliptic curve $E$, by identifying the pairs of points in $E\cap \ell_i-\{x_i\}$ for $1\leq i\leq \delta$. It is easy to see that
$\Pi_{\delta}$ is rational if $\delta \leq 9$. Clearly $f_{\delta}$ is dominant, just because every elliptic curve can be realized as a plane cubic. It follows that $B_{\delta + 1, \delta}^-$ is unirational when $\delta \leq 9$. \end{proof} \par \noindent
Unfortunately one cannot apply Theorem \ref{uni8} to the case $g = 9$, since the assumptions of Theorem \ref{assu} are satisfied only if $\delta \leq 5$.

\section{The Scorza curve}
This section serves as a preparation for the proof of Theorem \ref{degen} and we discuss in detail a correspondence $T_{\eta}\subset C\times C$ associated to each (non-vanishing) theta-characteristic
$[C, \eta]\in \cS_g^+-\thet$. This correspondence was used by G. Scorza \cite{Sc} to provide a birational isomorphism between $\cM_3$ and $\cS_3^+$
(see also \cite{DK}), and recently in \cite{TZ}, where several conditional statements of Scorza's have been rigourously established.

For a fixed theta-characteristic $[C, \eta]\in \cS_g^+-\thet$,  we consider the curve
$$T_{\eta}:=\bigl\{(x, y)\in C\times C: H^0(C, \eta\otimes \OO_C(x-y))\neq 0\bigr\}.$$
By Riemann-Roch, it follows that $T_{\eta}$ is a symmetric correspondence which misses the diagonal $\Delta\subset C\times C$. The curve $T_{\eta}$ has a natural fixed
point free involution and we denote by $f: T_{\eta}\rightarrow \Gamma_{\eta}$ the associated \'etale double covering. Under the assumption that $T_{\eta}$ is a reduced curve, its class is computed in \cite{DK} Proposition 7.1.5:
$$T_{\eta}\equiv (g-1)F_1+(g-1)F_2+\Delta.$$
Here $F_i\in H^2(C\times C, \mathbb Q)$ denotes the class of the fibre of the $i$-th projection $C\times C\rightarrow C$.
\begin{theorem}\label{scorza}
For a general theta-characteristic $[C, \eta]\in \cS_g^+$, the Scorza curve $T_{\eta}$ is a smooth curve of genus $g(T_{\eta})=3g(g-1)+1$.
\end{theorem}
\begin{proof} It is straightforward to show that a point $(x, y)\in T_{\eta}$ is singular if and only if
\begin{equation}\label{sing}
H^0(C, \eta\otimes \OO_C(x-2y))\neq 0 \ \mbox{ and
}H^0(C, \eta\otimes \OO_C(y-2x))\neq 0.
\end{equation}
 By induction on $g$, we show that for a general even spin curve such a pair $(x, y)$ cannot exist. We assume the result holds for a general
 $[C, \eta_C]\in \cS_{g-1}^+$. We fix a general point $q\in C$, an elliptic curve $D$ together with $\eta_D\in \mbox{Pic}^0(D)-\{\OO_D\}$ with $\eta_D^{\otimes 2}=\OO_D$
and consider the spin curve $t:=[C\cup E \cup D, \eta_{|C}=\eta_C,\ \eta_{| E}=\OO_E(1),\ \eta_{|D}=\eta_D]\in \ss_g^+$, obtained from $C\cup _q D$ by inserting an exceptional component $E$. Since the exceptional component plays no further role in the proof, we are going to suppress it.

We assume by contradiction that $t\in \ss_g^+$ lies in the closure of the locus of spin curves with singular Scorza curve. Then there exists a nodal
curve $C\cup_q D'$ semistably equivalent to $C\cup _q D$ obtained by inserting a possibly empty chain on $\PP^1$'s at the node $q$ (therefore, $p_a(D')=1$ and we may regard $D$ as a subcurve of $D'$), as well as
smooth points $x, y\in C\cup D'$ together with two limit linear series $\sigma=\{\sigma_C, \sigma_{D'}\}$ and $\tau=\{\tau_C, \tau_{D'}\}$ of type $\mathfrak g_{g-2}^0$ on $C\cup D'$ such that the underlying line bundles corresponding to $\sigma$ (resp. $\tau$) are uniquely determined twists
at the nodes of the line bundle $\eta\otimes \OO_{C\cup D'}(x-2y)$ (resp. $\eta\otimes \OO_{C\cup D'}(y-2x)$). The precise twists are determined by the limit linear series condition that each aspect of a limit $\mathfrak g_{g-2}^0$ have degree $g-2$. We distinguish three cases depending on which components of $C\cup D'$ the points $x$ and $y$ specialize.

\noindent
{\textbf{(i)}} \ $x, y\in C$. Then $\sigma_C\in H^0(C, \eta_C\otimes \OO_C(x-2y+q)), \tau_C\in H^0(C, \eta_C \otimes \OO_C(y-2x+q))$, while $\sigma_{D}, \tau_D\in H^0\bigl(D, \eta_D\otimes \OO_D((g-2)q)\bigr)$. Denoting by $\{q'\}\in D\cap \overline{(C\cup D')-D}$ the point where $D$ meets the rest of the curve, one has the compatibility conditions $$\mbox{ord}_q(\sigma_C)+\mbox{ord}_{q'}(\sigma_D)\geq g-2 \ \mbox{ } \mbox{ and  } \ \mbox{ord}_q(\tau_C)+\mbox{ord}_{q'}(\tau_D)\geq g-2,$$ which leads
to $\mbox{ord}_q(\sigma_C)\geq 1$ and $\mbox{ord}_q(\tau_C)\geq 1$, that is, we have found two points $x, y\in C$ such that $H^0(C, \eta_C(x-2y))\neq 0$ and $H^0(C, \eta_C(y-2x))\neq 0$, which contradicts the inductive assumption on $C$.

\noindent
{\textbf{(ii)}} \ $x, y\in D'$. This case does not appear if we choose $\eta_C$ such that $H^0(C, \eta_C)=0$. Indeed, for degree reason, both non-zero sections $\sigma_C, \tau_C$ must lie in the space $H^0(C, \eta_C)$.

\noindent
{\textbf{(iii)}} $x\in C, y\in D'$. For simplicity, we assume first that $y\in D$. We find that $$\sigma_C\in H^0(C, \eta_C\otimes \OO_C(x-q)), \ \sigma_D\in H^0(D, \eta_D\otimes \OO_D(g\cdot q'-2y))\ \mbox{ and }$$
 $$\tau_C\in H^0(C, \eta_C\otimes \OO_C(2q-2x)), \ \tau_D\in H^0(D, \eta_D\otimes \OO_C(y+(g-3)\cdot q')).$$
 We claim that $\mbox{ord}_q(\sigma_C)=\mbox{ord}_q(\tau_C)=0$ which can be achieved by a generic choice of $q\in C$. Then $\mbox{ord}_{q'}(\sigma_D)\geq g-2$, which implies that $\eta_D=\OO_D(2y-2q)$. Similarly, $\mbox{ord}_q(\tau_D)\geq g-2$ which
 yields that $\eta_D=\OO_D(q-y)$, that is, $\eta_D^{\otimes 3}=\OO_D$. Since $\eta_D$ was assumed to be a non-trivial point of order $2$
 this leads to a contradiction. Finally, the case $y\in D'-D$, that is, when $y$ lies on an exceptional subcurve $E'\subset D'$ is dealt with similarly: Since $\mbox{ord}_q(\sigma_C)=\mbox{ord}_q(\tau_C)=0$, by compatibility, after passing through the component $E'$, one obtains that
 $\mbox{ord}_{q'}(\sigma_D)\geq g-2$. Since $\sigma_D\in H^0(D, \eta_D\otimes \OO_D((g-2)q'))$ and $\eta_D\neq \OO_D$, we obtain a contradiction
\end{proof}

\section{The stack of degenerate odd theta-characteristics}
In this section we define a Deligne-Mumford stack $\textbf{X}_g\rightarrow \overline{\textbf{S}}_g^-$ parameterizing limit linear series
$\mathfrak g_{g-1}^0$ which appear as limits of degenerate theta-characteristics on smooth curves. The push-forward of $[\textbf{X}_g]$
is going to be precisely our divisor $\zz_g$. Having a good description of $\textbf{X}_g$ over the boundary will enable us to determine
all the coefficients in the expression of $[\zz_g]$ in $\mbox{Pic}(\ss_g^-)$ and thus prove Theorem \ref{degen}. We will use throughout the
test curves in $\ss_g^-$ constructed in Section 1.

We first define  a partial compactification
$\pem_g:=\textbf{M}_g \cup \widetilde{\Delta}_0 \cup \ldots  \cup \widetilde{\Delta}_{[g/2]}$ of $\overline{\textbf{M}}_g$,
obtaining by adding to $\textbf{M}_g$ the open sub-stack
$\widetilde{\Delta}_0\subset \Delta_0$ of one-nodal irreducible curves
$[C_{yq}:=C/y\sim q]$, where $[C, y, q]\in \cM_{g-1, 2}$ is a Brill-Noether general curve together with their degenerations $[C\cup D_{\infty}]$ where $D_{\infty}$ is an elliptic curve with $j(D_{\infty})=\infty$,  as well as the open substacks
$\widetilde{\Delta}_j \subset \Delta_j$ for $1\leq j\leq [g/2]$ classifying curves $[C\cup_y D]$, where $[C]\in \cM_{j}$ and $[D]\in \cM_{g-j}$ are  Brill-Noether general curves in the respective moduli spaces.
 Let
$p:\pem_{g, 1}\rightarrow \pem_g$ be the
universal curve. We  denote
$\pes_g^-:=\pi^{-1}(\pem_g)\subset \overline{\textbf{S}}_g^-$ and
note that for all $0\leq j\leq [g/2]$ the boundary divisors
$A_j':=A_j\cap \ps_g^-, \
B_j':=B_j\cap \ps_g^-$
are mutually disjoint inside $\ps_g^-$. Finally, we consider
$\cZ:=\pes_g^-\times _{\pem_{g}} \pem_{g, 1}$ and denote by
$p_1:\cZ\rightarrow \pes_g^-$ the projection.

Following the local description of the projection $\overline{\textbf{S}}_g^-\rightarrow
\overline{\textbf{M}}_g$ carried out in \cite{C}, in order to obtain the universal spin curve over $\pes_g^-$
one has first to
blow-up the codimension $2$ locus $V\subset \cZ$ corresponding to
points
$$v=\Bigl(\bigl[C\cup_{\{y, q\}} E, \eta_C^{\otimes 2}=K_C, h^0(\eta_C)\equiv 1\mbox{ }  \mathrm{ mod }\mbox{ }  2, \mbox{ } \eta_E=\OO_E(1)\bigr],
\  \nu(y)=\nu(q)\Bigr) \in B_0'\times _{\pem_
g} \pem_{g, 1}$$ (recall that  $\nu:C\rightarrow C_{yq}$ denotes
the normalization map, so $v$ corresponds to the marked point specializing to the node of the curve $C_{yq}$). Suppose that $(\tau_1, \ldots, \tau_{3g-3})$ are
local coordinates in an \'etale neighbourhood of $[C\cup_{\{y, q\}}
E, \eta_C, \eta_E]\in \ps_g^-$, such that the local equation of
the divisor $B_0'$ is $(\tau_1=0)$. Then $\cZ$  around $v$
admits local coordinates $(x, y, \tau_1, \ldots, \tau_{3g-3})$ verifying
the equation $xy=\tau_1^2$,  in particular, $\cZ$ is singular along $V$.
Next, for $1\leq j\leq [g/2]$ one blows-up the codimension $2$ loci
$V_j\subset \cZ$ consisting of points
$$\Bigl(\bigr[C\cup_q D, \eta_C, \eta_D\bigr], \ q\in C\cap D\Bigr)\in (A_j'\cup B_j')\times_{\pem _g}\pem_{g, 1}.$$
This corresponds to inserting an exceptional component in each spin curve in $\pi^*(\widetilde{\Delta}_j)$.
We denote by $$\mathcal{C}:=\mathrm{Bl}_{V\cup V_1\cup \ldots \cup V_{[g/2]}}(\cZ)$$ and  by
$f:\mathcal{C} \rightarrow \pes_g^-$ the induced family of spin curves.
Then for every $[X, \eta, \beta]\in \ps_g^-$ we have an isomorphism between $f^{-1}([X,
\eta, \beta])$ and the quasi-stable curve $X$.

There exists a spin line bundle $\P \in
\mathrm{Pic}(\mathcal{C})$ of relative degree $g-1$ as well as a morphism of $\OO_{\mathcal C}$-modules
$B:\P^{\otimes 2}\rightarrow \omega_f$ having the property
that $\P_{| f^{-1}([X, \eta, \beta])}=\eta$ and $B_{| f^{-1}([X,
\eta, \beta])}=\beta:\eta^{\otimes 2}\rightarrow \omega_X$, for all
spin curves $[X, \eta, \beta]\in \ps_g^-$. We note that for the even moduli
space $\ps_g^+$ one has an analogous construction of the universal spin curve.

Next we define the stack $\tau:\textbf{X}_g\rightarrow \pes_g^-$ classifying limit $\mathfrak g_{g-1}^0$ which are twists of degenerate odd-spin curves. For a tree-like curve $X$ we denote by $\overline{G}^r_d(X)$ the scheme of limit linear series $\mathfrak g^r_d$.
The fibres of the morphism $\tau$ have the following description:

\noindent $\bullet$ $\tau^{-1}(\textbf{S}_g^-)$ parameterizes triples $\bigl([C, \eta], \sigma, x\bigr)$, where $[C, \eta]\in \cS_g^-, \ x\in C$ is a point and $\sigma \in \PP H^0(C, \eta)$ is a section such that $\mbox{div}(\sigma)\geq 2x$.

\noindent $\bullet$  For $1\leq j\leq [g/2]$ the inverse image $\tau^{-1}(A_j'\cup B_j')$ parameterizes elements of the form
$$\Bigl(X, \sigma\in \overline{G}_{g-1}^0(X), \ x\in X_{\mathrm{reg}} \Bigr),$$ where
$(X, x)$ is a $1$-pointed quasi-stable curve semistably equivalent to the underlying curve of a spin curve $[C\cup_q E \cup _{q'} D, \ \eta_C, \eta_E, \eta_D]\in A_j'\cup B_j'$, with $E$ denoting the exceptional component, $g(C)=j, \ g(D)=g-j, \{q\}=C\cap E, \{q'\}=E\cap D$ and
$$\sigma_C\in \PP H^0\bigl(C, \eta_C\otimes \OO_C((g-j)q)), \ \sigma_D\in \PP H^0\bigl(D, \eta_D\otimes \OO_D(jq')\bigr), \sigma_E\in \PP H^0(E, \OO_E(g-1))$$ are aspects of the limit linear series $\sigma$ on $X$. Moreover, we require that  $\mbox{ord}_x(\sigma)\geq 2$.

\noindent $\bullet$ $\tau^{-1}(B_0')$ parameterizes elements $\bigl(X, \ \eta\in \mathrm{Pic}^{g-1}(X), \ \sigma\in \PP H^0(X, \eta),\  x\in X_{\mathrm{reg}}\bigr)$, where $(X, x)$ is a $1$-pointed quasi-stable curve equivalent to the curve underlying a point
$[C\cup_{\{y, q\}} E, \eta_C, \ \eta_E]\in B_0'$, the line bundle $\eta$ on $X$ satisfies $\eta_{|C}=\eta_C$ and $\eta_{| E}=\eta_E$ and $\eta_{| Z}=\OO_Z$ for the remaining  components of $X$. Finally, we require $\mbox{ord}_x(\sigma)\geq 2$.

\noindent $\bullet$ $\tau^{-1}(A_0')$ corresponds to points $\bigl(X, \eta \in \mathrm{Pic}^{g-1}(X), \sigma\in \PP H^0(X, \eta),\ x\in X_{\mathrm{reg}}\bigr)$, where $(X, x)$ is a $1$-pointed quasi-stable curve equivalent to the curve underlying a point $[C_{yq}, \eta_{C_{yq}}]\in A_0'$,  and if $\mu: X\rightarrow C_{yq}$ is the map contracting all exceptional components, then $\mu^*(\eta_{C_{yq}})=\eta$ (in particular $\eta$ is trivial along exceptional components), and finally $\mbox{ord}_x(\sigma) \geq 2$.

Using general constructions of stacks of limit linear series  cf. \cite{EH1}, \cite{F2}, it is clear that $\textbf{X}_g$ is a Deligne-Mumford stack. There exists a proper morphism $$\tau: \textbf{X}_g\rightarrow \pes_g^-$$ that factors through the universal curve and we denote by $\chi:\textbf{X}_g\rightarrow \cC$ the induced morphism, hence $\tau=f\circ \chi$. The push-forward of the coarse moduli space $\tau_*([\mathcal{X}_g])$ equals scheme-theoretically
$\zz_g\cap \ps_g^-$. It appears possible to extend $\textbf{X}_g$ over the entire $\overline{\textbf{S}}_g^-$ but this is not necessary in order to prove Theorem \ref{kodaira} and we skip the details.

We are now in a position to calculate the class of the divisor $\zz_g$ and we expand its class in the Picard group of $\ss_g^-$
\begin{equation}\label{expansion}
\zz_g\equiv \bar{\lambda}\cdot \lambda-\bar{\alpha_0}\cdot \alpha_0-\bar{\beta_0}\cdot \beta_0-\sum_{i=1}^{[g/2]} \bar{\alpha}_i\cdot \alpha_i-\sum_{i=1}^{[g/2]} \bar{\beta_i}\cdot \beta_i\in \mathrm{Pic}(\ss_g^-),
\end{equation}
where $\bar{\lambda}, \bar{\alpha_i}, \bar{\beta_i}\in \mathbb Q$ for $i=0, \ldots, [g/2]$.
We start by determining the coefficients of the divisors $\alpha_i$ and $\beta_i$ for $1\leq i\leq [g/2]$.

\begin{proposition} For $1\leq i\leq [g/2]$ we have that $F_i\cdot \zz_{g}=4(g-i)(i-1)$ and the intersection is everywhere transverse.
It follows that $\bar{\alpha_i}=2(g-i)$.
\end{proposition}
\begin{proof} We recall from the definition of $F_i$ that we have fixed theta-characteristics of opposite parity  $\eta_C^-\in \mbox{Pic}^{i-1}(C)$ and
$\eta_D^+\in \mbox{Pic}^{g-i-1}(D)$. Choose a point $t=(X, \eta, \sigma, x)\in \tau^{-1}(F_i)$. It is a simple exercise to show that the "double" point $x$ of $\sigma\in \overline{G}^0_{g-1}(X)$ cannot specialize to the exceptional component, therefore one has only two cases to consider depending on whether $x$ lies on $C$ or on $D$. Assume first that $x\in C$ and then $\sigma_C\in \PP H^0(C, \eta_C^-\otimes \OO_C((g-i)q))$ and $\sigma_D \in \PP H^0(D, \eta_D^+ \otimes \OO_D(iq))$, where $\{q\}=C\cap D$
is a point which moves on $C$ but is fixed on $D$.
Then $\mbox{ord}_q(\sigma_D)\leq i-1$, therefore $\mbox{ord}_q(\sigma_C)\geq g-i$ and then $\sigma_C(-(g-i)q)\in \PP H^0(C, \eta_C^-)$. In particular, if we choose $[C, \eta_C^-]\in \cS_{i}-\cZ_i$, then the section $\sigma_C(-(g-i)q)$ has only simple zeros, which shows that $x$ cannot lie on $C$, so this case does not occur.

We are left with the possibility  $x\in D-\{q\}$. One observes that  $\mbox{ord}_q(\sigma_C)=g-i+1$ and
$\mbox{ord}_q(\sigma_D)=i-2$. In particular, $q\in \mathrm{supp}(\eta_C^-)$ which gives $i-1$ choices for the moving point $q\in C$. Furthermore $\sigma_D(-(i-2)q)\in H^0(D, \eta_D^+\otimes \OO_D(2q-2x))$, that is, $x$ specializes to one of the ramification points of the pencil $\eta_D^+\otimes \OO_D(2q)\in W_{g-i+1}^1(D)$. We note that because of the generality of $[D, \eta_D^+]\in \cS_{g-i}^+$ as well as that of $q\in D$, the pencil is base point free and complete. From the Hurwitz-Zeuthen formula one finds $4(g-i)$ ramification points of $|\eta_D^+\otimes \OO_D(2q)|$, which leads to the formula $F_i\cdot \zz_g=4(g-i)(i-1)$. The fact that $\tau_*(\textbf{X}_g)$ is transverse to $F_i$ follows because the formation of $\textbf{X}_g$ commutes with restriction to $B_0'$ and then  one can easily show in a way similar to \cite{EH2} Lemma 3.4, or by direct calculation that $\textbf{X}_g\times _{\pes_g^-} B_0'$ is smooth at any of the points in $\tau^{-1}(F_i)$.
\end{proof}

\begin{proposition} For $1\leq i\leq [g/2]$ we have that $G_i\cdot \zz_g=4i(i-1)$ and the intersection is transversal. In particular $\bar{\beta_i}=2i$.
\end{proposition}
\begin{proof} This time we fix general points $[C, \eta_C^+]\in \cS_{i}^+$ and $[D, \eta_D^-]\in \cS_{g-i}^-$ and $q\in C\cap D$ which is a fixed general point on $D$ but an arbitrary point on $C$. Again, it is easy to see that if $t=(X, \sigma, x)\in \tau^{-1}(G_i)$ then $x$ must lie either on $C$ or on $D$. Assume first that $x\in C-\{q\}$. Then the aspects of $\sigma$ are described as follows
$$\sigma_C\in \PP H^0(C, \eta_C^+\otimes \OO_C((g-i)q)),\ \ \sigma_D\in \PP H^0(D, \eta_D^-\otimes \OO_D(iq))$$
and moreover $\mbox{ord}_x(\sigma_C)\geq 2$. The point $q\in D$ can be chosen so that it does not lie in $\mbox{supp}(\eta_D^-)$, hence $\mbox{ord}_q(\sigma_D)\leq i$ and then $\mbox{ord}_q(\sigma_C)\geq g-i-1$.
This leads to the conclusion $H^0(C, \eta_C^+\otimes \OO_C(y-2x))\neq 0$, or equivalently $(x, y)\in C\times C$ is a ramification point of the degree $i$ covering
$p_1:T_{\eta_C^+}\rightarrow C$ from the associated Scorza curve. We have shown that $T_{\eta_C^+}$ is smooth of genus $1+3i(i-1)$ (cf. Theorem \ref{scorza}) and moreover all the ramification points of $p_1$ are ordinary, therefore we find
$$\mbox{deg } \mbox{Ram}({p_1})=2g(T_{C_{\eta_C^+}})-2-\mbox{deg}(p_1)(2i-2)=4i(i-1)$$
choices when $x\in C$. Next possibility is $x\in D-\{q\}$. The same reasoning as above shows that $\mbox{ord}_q(\sigma_C)\leq g-i-1$, therefore
$\mbox{ord}_q(\sigma_D)\geq i$ as well as $\mbox{ord}_x(\sigma_D)\geq 2$. Since $\sigma_D(-iq)\in \PP H^0(D, \eta_D^-)$, this case does not occur if $[D, \eta_D^-]\in \cS_{g-i}^--\cZ_{g-i}$.
\end{proof}
Next we prove that $\zz_g$ is disjoint from both elliptic pencils $F_0$ and $G_0$:
\begin{proposition}\label{f0}
We have that $F_0\cdot \zz_g=0$ and $G_0\cdot \zz_g=0$. The equalities $\bar{\alpha}-12\bar{\alpha_0}+\bar{\alpha_1}=0$ and $3\bar{\alpha}-12\bar{\alpha}_0-12\bar{\beta_0}+3\bar{\beta}_1=0$ follow.
\end{proposition}
\begin{proof} We first show that $F_0\cap \zz_g=\emptyset$ and we assume by contradiction that there exists $t=(X, \sigma, x)\in \tau^{-1}(F_0)$.
Let us deal first with the case when $st(X)=C\cap E_{\lambda}$, with $E_{\lambda}$ being a smooth curve of genus $1$. The key point is that
the point of attachment $q\in C\cap E_{\lambda}$ being general, we can assume that $(x, q)\notin \mbox{Ram}\{p_1: T_{\eta_C^+}\rightarrow C\}$, for all
$x\in C$. This implies that $H^0(C, \eta_C^{+}\otimes \OO_C(q-2x))=0$ for all $x\in C$, therefore a section $\sigma_C\in \PP H^0(C, \eta_C^{+}\otimes \OO_C(q))$ cannot vanish twice anywhere. Thus either $x\in E_{\lambda}-\{q\}$ or $x$ lies on some exceptional component of $X$. In the former case, since $\mbox{ord}_q(\sigma_C)=0$, it follows that $\mbox{ord}_q(\sigma_{E_{\lambda}})\geq g-1$, that is, $\sigma_{E_{\lambda}}$ has no zeroes other than $q$ (simple or otherwise). In the latter case, when $x\in E$, with $E$ being an exceptional component, we denote by $q'\in E$ the point of intersection of $E$ with the connected subcurve of $X$ containing $C$ as a subcomponent. Since as above, $\mbox{ord}_q(\sigma_C)=0$, by compatibility
it follows that $\mbox{ord}_{q'}(\sigma_E)=g-1$. But $\sigma_E\in \PP H^0(E, \OO_E(g-1))$, that is, $\sigma_E$ does not vanish at $x$, a contradiction. The proof that $G_0\cap \zz_g=\emptyset$ is similar and we omit the details.
\end{proof}

The trickiest part in the calculation of $[\zz_f]$ is the computation of the following intersection number:
\begin{proposition}\label{h0}
If $H_0\subset B_0$ is the covering family lying in the ramification divisor of $\ss_g^-$, then one has that
$H_0\cdot \zz_g=2(g-2)$ and the intersection consists of $g-2$ points each counted with multiplicity $2$. Therefore the relation $(g-1)\bar{\beta_0}-\bar{\beta_1}=2(g-2)$ holds.
\end{proposition}
\begin{proof}
We first describe the set-theoretic intersection $\tau_*(\mathcal{X}_g)\cap H_0$. We recall that we have fixed $[C, q, \eta_C^-]\in \cS_{g-1, 1}^-$  and start by choosing a point
$t=(X, \eta, \sigma, x)\in \tau^{-1}(H_0)$. Assume first that $X=C\cup_{\{y, q\}} E$, where $y\in C$, that is, $x$ does not specialize to one of the nodes of $C\cup E$. Suppose first that $x\in C-\{y, q\}$. From the Mayer-Vietoris sequence on $X$ we write
$$0\neq \sigma \in H^0(X, \eta\otimes \OO_X(-2x))=\mbox{Ker}\bigl\{H^0(C, \eta_C^-\otimes \OO_C(-2x))\oplus H^0(E, \OO_E(1))\stackrel{\mathrm{ev}_{y, q}}\longrightarrow \mathbb C_{\{y, q\}}^2\bigr\},$$
we obtain that $H^0(C, \eta_C^-\otimes \OO_C(-2x))\neq 0$. This case can be avoided by choosing $[C, \eta_C^-]\in \cS_{g-1}^--\cZ_{g-1}$.

\vskip 3pt

Next we consider  the possibility $x\in E-\{y, q\}$. The same Mayer-Vietoris argument reads in this case $0\neq \mbox{Ker}\bigl\{H^0(C, \eta_C^-)\oplus H^0(E, \OO_E(-1))\stackrel{ev_{y, q}}\longrightarrow \mathbb C_{\{y, q\}}^2\bigr\}$, that is, $y+q\in \mbox{supp}(\eta_C^-)$. This case can be avoided as well by starting with a general point $q\in C- \mbox{supp}(\eta_C^-)$. Thus the only possibility is that $x$ specializes to one of the nodes $y$ or $q$.

We deal first with the case when $x$ and $q$ coalesce and there is no loss of generality in assuming that $X=C\cup E\cup E'$, where both components $E$ and $E'$ are copies of $\PP^1$ and $C\cap E=\{y\}, C\cap E'=\{q\}, E\cap E'=\{y'\}$ and moreover $x\in E'-\{y', q\}$. The restrictions of the line bundle $\eta\in \mbox{Pic}^{g-1}(X)$
are such that $\eta_{| C}=\eta_C^-, \eta_E=\OO_E(1)$ and $\eta_{E'}=\OO_{E'}$. We write
$$0\neq \sigma=(\sigma_C, \sigma_E, \sigma_{E'})\in \mbox{Ker}\bigl\{H^0(C, \eta_C^-)\oplus H^0(E, \OO_E(1))\oplus H^0(E', \OO_{E'}(1))\stackrel{ev_{y, y', q}}\longrightarrow \mathbb C_{y, y', q}\bigr\},$$
hence $\sigma_{E'}=0$, and then by compatibility $\sigma_C(q)=0$, that is, $q\in \mbox{supp}(\eta_C^-)$ and again this case can be ruled out by a suitable choice of $q$. The last possible situation is when $x$ and the moving point $y\in C$ coalesce, in which case $X=C\cup E\cup E'$, where this time $C\cap E=\{q\}, C\cap E'=\{y\}, E\cap E'=\{y'\}$ and again $x\in E'-\{y', q\}$. Writing one last time the Mayer-Vietoris sequence we find
that $\sigma_{E'}=0$ and then $\sigma_{E}(y')=0$ and $\sigma_C(y)=0$, that is, $y\in \mbox{supp}(\eta_C^-)$ and then $\sigma_C$ is uniquely determined up to a constant. Finally $\sigma_E\in  H^0(E, \OO_E(1)(-y'))$ is uniquely specified by the gluing condition $\sigma_E(q)=\sigma_C(q)$.
All in all, $H_0\cap \zz_g=|\mbox{supp}(\eta_C^-)|=g-2$.

This discussion  singles out an irreducible component $\Xi\subset \chi_*(\mathcal{X}_g)\subset \cC$ of the intersection $\chi(\mathcal{X}_g)\cap f^{-1}(B_0')$, namely
$$\Xi=\Bigl\{\bigl([C\cup_{\{y, q\}} E, \eta_C, \eta_E], x): y\in \mbox{supp}(\eta_C^-) \mbox{ and } \  \ x=y\in X_{\mathrm{sing}}\Bigr\},$$
where recall that $f:\cC\rightarrow \pes_g^-$ is the universal spin curve. Since $\Xi\subset \mathrm{Sing}\bigl(\chi_*(\mathcal{X}_g)\bigr)$, after a simple local analysis, it follows that each point in $\tau^{-1}(H_0)$ occurs counted with multiplicity $2$.
\end{proof}
\begin{remark} A partial independent check of Theorem \ref{degen} is obtained by using the Porteous formula to determine the coefficient $\bar{\lambda}$ in the expression of $[\zz_g]$. By abuse of notation we still denote by $f:\cC\rightarrow \textbf{S}_g^-$ the restriction of the universal spin curve to the locus of smooth curves and $\eta\in \mathrm{Pic}(\cC)$ the spin bundle of relative degree $g-1$. Then $\cZ_g$ is the push-forward via $f:\cC\rightarrow \textbf{S}_g^-$ of the degeneration locus of the sheaf  morphism $\phi:f_*(\eta)\rightarrow J_1(\eta)$  (both these sheaves are locally free away a subset of codimension $3$ in $\textbf{S}_g^-$ and throwing away this locus has no influence on divisor class calculations). Since $\mathrm{det}(f_*\eta)=(f_*\eta)^{\otimes 2}$, it follows that $c_1\bigl(f_*(\eta)\bigr)=-\lambda/4$, whereas the Chern classes of the first jet bundle $J_1(\eta)$ are calculated using the standard exact sequence on $\cC$
$$0\longrightarrow \eta\otimes \omega_{f}\longrightarrow J_1(\eta)\longrightarrow \eta\longrightarrow 0.$$
Remembering Mumford's formula  $f_*(c_1^2(\omega_f))=12\lambda$, one finally writes that $$[\cZ_g]=f_* c_2\Bigl(J_1(\eta)-f_*(\eta)\Bigr)=f_*\Bigl(\frac{3}{4}c_1(\omega_f)^2-2c_1(\omega_f) \cdot c_1(f_*(\eta))\Bigr)=(g+8)\ \lambda \in \mathrm{Pic}(\textbf{S}_g^-).$$
\end{remark}

\section{A divisor of small slope on $\mm_{12}$}

The aim of this section is to construct an effective divisor $D\in \mbox{Eff}(\mm_{12})$ of slope
$s(D)<6+12/13$, that is, violating the Slope Conjecture. As pointed out in the proof of Theorem \ref{kodaira}, this is precisely what is required
in order to show that $\ss_{12}^-$ is a variety of  general type.

\begin{theorem}\label{m12} The following locus consisting of curves of genus $12$
$$\mathfrak{D}_{12}:=\{[C]\in \cM_{12}: \exists L\in W^4_{14}(C) \ \mbox{ with } \ \mathrm{Sym}^2 H^0(C, L)\stackrel{\mu_0(L)}\longrightarrow H^0(C, L^{\otimes 2}) \mbox{ not injective}\}$$ is a divisor on $\cM_{12}$. The class of its compactification inside $\mm_{12}$ equals
$$\overline{\mathfrak{D}}_{12}\equiv 13245\ \lambda-1926\ \delta_0-9867\ \delta_1-\sum_{j=2}^6 b_j\ \delta_j\in \mathrm{Pic}(\mm_{12}),$$
where $b_j\geq b_1$ for $j\geq 2$. In particular, $s(\overline{\mathfrak{D}}_{12})=\frac{4415}{642}<6+\frac{12}{13}$.
\end{theorem}
This implies the following upper bound for the slope $s(\mm_{12})$ of the moduli space:

\begin{corollary}\label{slope12}
$$6+\frac{10}{12}\leq s(\mm_{12}):=\mathrm{inf}_{D\in \mathrm{Eff}(\mm_{12})} s(D) \leq \frac{4415}{642}\ \Bigl(=6+\frac{10}{12}+\frac{14}{321}\Bigr).$$
\end{corollary}
Another immediate application, via \cite{Log}, \cite{F1}, concerns the birational type of the moduli space $\mm_{g, n}$ of $n$-pointed stable curves of genus $g$:
\begin{theorem}\label{gentype12}
The moduli space of $n$-pointed curves $\mm_{12, n}$ is of general type for $n\geq 11$.
\end{theorem}

The divisor $\mathfrak{D}_{12}$ is constructed as the push-forward of a codimension $3$ cycle in the stack $\mathfrak{G}^4_{14}\rightarrow \textbf{M}_{12}$ classifying linear series $\mathfrak g^4_{14}$. We describe the construction of this cycle, then extend this determinantal structure over a partial compactification of $\cM_{12}$. This will be essential to understand the intersection of $\overline{\mathfrak{D}}_{12}$ with the boundary divisors $\Delta_0$ and $\Delta_1$ of $\mm_{12}$.
We denote by $\textbf{M}_{12}^p$ the open substack of $\textbf{M}_{12}$ consisting
of curves $[C]\in \cM_{12}$ such that $W^4_{13}(C)=\emptyset$ and
$W^5_{14}(C)=\emptyset$. Results in Brill-Noether theory
guarantee that $\mbox{codim}(\cM_{12}-\cM_{12}^p, \cM_{12})\geq 3$.
If $\mathfrak{Pic}^{14}_{12}$ denotes the Picard stack of degree
$14$ over $\textbf{M}_{12}^p$, then we consider the smooth Deligne-Mumford substack
$\mathfrak{G}^4_{14}\subset \mathfrak{Pic}^{14}_{12}$ parameterizing
pairs $[C, L]$, where $[C]\in \cM_{12}^p$ and $L\in W^4_{14}(C)$ is a (necessarily complete and base point free)
linear series. We
denote by $\sigma: \mathfrak{G}^4_{14}\rightarrow \textbf{M}_{12}^p$ the
forgetful morphism. For a general $[C]\in \cM_{12}^p$, the fibre
$\sigma^{-1}([C])=W^4_{14}(C)$ is a smooth surface.
\vskip 4pt

Let $\pi:\textbf{M}_{12, 1}^p \rightarrow \textbf{M}_{12}^p$ be the universal
curve and then $p_2:\textbf{M}_{12, 1}^p\times
_{\textbf{M}_{12}^p} \mathfrak{G}^4_{14}\rightarrow \mathfrak{G}^4_{14}$
denotes the natural projection. If $\mathcal{L}$ is a Poincar\'e bundle
over $\textbf{M}_{12, 1}^p\times_{\textbf{M}_{12}^p} \mathfrak{G}^4_{14}$ (or over an \'etale cover of it), then by
Grauert's Theorem, both  $$\E:=(p_2)_*(\mathcal{L})  \ \mbox{ and }
\F:=(p_2)_*(\mathcal{L}^{\otimes 2})$$ are vector bundles over
$\mathfrak{G}^4_{14}$, with $\mbox{rank}(\E)=5$ and
$\mbox{rank}(\F)=h^0(C, L^{\otimes 2})=17$ respectively. There is a natural vector bundle morphism over $\mathfrak G^4_{14}$ given by multiplication of sections,
$$\phi:\mbox{Sym}^{2}(\E)\rightarrow \F,$$ and we denote by
$\cU_{12}\subset \mathfrak{G}^4_{14}$ its first degeneracy locus. We
set $\mathfrak{D}_{12}:=\sigma_*(\cU_{12})$. Since the degeneracy locus $\cU_{12}$
has expected codimension $3$ inside $\mathfrak{G}^4_{14}$, the locus
$\mathfrak{D}_{12}$ is a virtual divisor on $\cM_{12}^p$.

We  extend the vector bundles $\E$ and $\F$ over a partial
compactification of $\mathfrak{G}^4_{14}$ given by limit $\mathfrak g^4_{14}$. We denote by
$\Delta_1^p\subset \Delta_1\subset \mm_{12}$ the locus of curves
$[C\cup_y E]$, where $E$ is an arbitrary elliptic curve, $[C]\in
\cM_{11}$ is a Brill-Noether general curve and $y\in
C$ is an arbitrary point. We then denote by $\Delta_0^p\subset
\Delta_0\subset \mm_{12}$ the locus consisting of curves $[C_{yq}]\in \Delta_0$, where $[C, q]\in \cM_{11, 1}$ is Brill-Noether general and $y\in
C$ is arbitrary, as well as their degenerations $[C\cup_q E_{\infty}]$ where
$E_{\infty}$ is a rational nodal curve.
Once we set
$$\ttem_{12}:=\tem_{12} \cup \Delta_0^p\cup \Delta_1^p\subset \rem_{12},$$
we can extend the morphism $\sigma$ to a proper morphism
$$\sigma:\widetilde{\mathfrak{G}}^4_{14}\rightarrow
\ttem_{12},$$ from the stack
$\widetilde{\mathfrak{G}}^4_{14}$ of limit linear
series $\mathfrak g^4_{14}$ over the partial compactification  $\ttem_{12}$ of $\textbf{M}_{12}$.

We extend the vector bundles $\E$ and $\F$ over the
stack $\widetilde{\mathfrak G}^4_{14}$. The proof of the following result proceeds along the lines of the proof of
Proposition 3.9 in \cite{F1}:
\begin{proposition}\label{vectbundles}
There exist two vector bundles $\E$ and $\F$ defined over
$\widetilde{\mathfrak G}^4_{14}$ with $\rm{rank}$$(\E)=5$ and
$\mathrm{rank}(\F)=17$, together with a vector bundle morphism $\phi:\mathrm{Sym}^2(\E)\rightarrow \F$,
such that the following statements hold:
\begin{itemize}
\item For $[C, L]\in \mathfrak{G}^4_{14}$, with $[C]\in \cM_{12}^p$, we have that
$$\E(C, L)=H^0(C, L)\ \mbox{ and } \ \F(C, L)=H^0(C, L^{\otimes 2}).$$
\item For $t=(C\cup_y E, l_C, l_E)\in \sigma^{-1}(\Delta_1^p)$,
where $g(C)=11, g(E)=1$ and $l_C=|L_C|$ is such that  $L_C\in
W^4_{14}(C)$ has a cusp at $y\in C$, then $\E(t)=H^0(C, L_C)$ and
$$\F(t)=H^0(C, L_C^{\otimes 2}(-2y))\oplus \mathbb C\cdot u^2,$$ where
$u\in H^0(C, L_C)$ is any section such that $\rm{ord}$$_y(u)=0$.
If $L_C$ has a base point at $y$, then $\E(t)=H^0(C, L_C)=H^0(C, L_C\otimes \OO_C(-y))$ and the image of a natural
map $\F(t)\rightarrow H^0(C, L_C^{\otimes 2})$ is the subspace $H^0(C, L_C^{\otimes 2}\otimes \OO_C(-2y))$.
\item Fix $t=[C_{yq}:=C/y\sim q, L] \in \sigma^{-1}(\Delta_0^p)$, with $q,
y\in C$ and $L\in \overline{W}^4_{14}(C_{yq})$  such that $h^0(C,
\nu^*L\otimes \OO_C(-y-q))=4$, where $\nu:C\rightarrow C_{yq}$ is the
normalization map. In the case when $L$ is locally free we have that
$$\E(t)=H^0(C, \nu^*L)\ \mbox{ and }\ \F(t)=H^0(C, \nu^*L^{\otimes
2}\otimes \OO_C(-y-q))\oplus \mathbb C\cdot u^2,$$ where $u\in
H^0(C, \nu^*L)$ is any section not vanishing at $y$ and $q$. In the
case when $L$ is not locally free, that is,  $L\in
\overline{W}_{14}^4(C_{yq})-W_{14}^4(C_{yq})$, then $L=\nu_*(A)$,
where $A\in W^4_{13}(C)$ and the image of the natural map
$\F(t)\rightarrow H^0(C, \nu^*L^{\otimes 2})$ is the subspace
$H^0(C, A^{\otimes 2})$.
\end{itemize}
\end{proposition}

To determine the push-forward $[\overline{\mathfrak{D}}_{12}]^{virt}=\sigma_*\bigl(c_3(\F-\mbox{Sym}^2(\E)\bigr)\in A^1(\cM_{12}^p)$, we study the restriction of the morphism $\phi$ along the pull-backs of two curves sitting in the boundary of
$\mm_{12}$ and which are defined as follows: We fix a general pointed curve $[C, q]\in \cM_{11, 1}$
and a general elliptic curve $[E, y]\in \cM_{1, 1}$. Then we consider
the families
$$C_0:=\{C/y\sim q: y\in C\}\subset \Delta_0^p \subset \mm_{12} \mbox{ and }
C_1:=\{C\cup_y E:y \in C\}\subset \Delta_1^p\subset \mm_{12}.
$$
These curves intersect the generators of $\mbox{Pic}(\mm_{12})$ as
follows:
$$C_0\cdot \lambda=0,\  C_0\cdot \delta_0=\mbox{deg}(\omega_{C_{yq}})=-22,\ C_0\cdot \delta_1=1
\mbox{ and } C_0\cdot \delta_j=0 \mbox{ for }2\leq j\leq  6, \mbox{
and }$$
$$C_1\cdot \lambda=0, \ C_1\cdot \delta_0=0, \ C_1\cdot \delta_1=-\mbox{deg}(K_C)=-20 \mbox{ and } C_1\cdot
\delta_j=0 \mbox{ for }2\leq j\leq 6.$$

Next, we fix a general pointed curve $[C, q]\in \cM_{11, 1}$ and describe the geometry of the pull-back $\sigma^*(C_0)\subset \widetilde{\mathfrak G}^4_{14}$. We consider the determinantal $3$-fold
$$Y:=\{(y, L)\in C\times W^4_{14}(C): h^0(C, L\otimes \OO_C(-y-q))=4\}$$ together with the projection $\pi_1:Y\rightarrow C$. Inside $Y$ we consider the following  divisors
$$\Gamma_1:=\{(y, A\otimes \OO_C(y)): y\in C, \ A\in W^4_{13}(C)\}
\ \mbox{  and } $$
$$\Gamma_2:=\{(y, A\otimes \OO_C(q)): y\in C, \ A\in
W^4_{13}(C)\}$$ intersecting transversally along the curve
$\Gamma:=\{(q, A\otimes \OO_C(q)): A\in W_{13}^4(C)\}\cong W_{13}^4(C).$
We introduce the blow-up $Y'\rightarrow Y$ of $Y$ along $\Gamma$ and denote by $E_{\Gamma} \subset Y'$ the
exceptional divisor
and by $\widetilde{\Gamma}_1, \widetilde{\Gamma}_2\subset Y'$ the
strict transforms of $\Gamma_1$ and $\Gamma_2$ respectively. We then define
$\widetilde{Y}:=Y'/\widetilde{\Gamma}_1\cong \widetilde{\Gamma}_2$,
to be the variety obtained from $Y'$ by identifying the divisors
$\widetilde{\Gamma}_1$ and $\widetilde{\Gamma}_2$ over each
$(y, A)\in C\times W^4_{13}(C)$. Let $\epsilon:\widetilde{Y}\rightarrow Y$ be the projection map.

\begin{proposition}\label{limitlin0}
With notation as above, one has a birational morphism
of $3$-folds
$$f: \sigma^*(C_0)\rightarrow \widetilde{Y},$$
which is an isomorphism outside a curve contained in $\epsilon^{-1}(\pi_1^{-1}(q))$. The map $f_{| (\pi_1 \epsilon f)^{-1}(q)}$ corresponds to forgetting the
$E_{\infty}$-aspect of each limit linear series. Accordingly, the vector bundles $\E_{| \sigma^*(C_0)}$ and $\F_{| \sigma^*(C_0)}$
are pull-backs under $\epsilon \circ f$ of vector bundles on $Y$.
\end{proposition}
\begin{proof}
We fix a point $y\in C-\{q\}$ and denote  by $\nu:C\rightarrow C_{yq}$ the normalization map, with $\nu(y)=\nu(q)$.
 We
investigate the variety $\overline{W}^4_{14}(C_{yq})\subset
\overline{\mbox{Pic}}^{14}(C_{yq})$ of torsion-free sheaves $L$ on
$C_{yq}$ with $\mbox{deg}(L)=14$ and $h^0(C_{yq}, L)\geq 5$. A locally free $L\in \overline{W}^4_{14}(C_{yq})$ is determined by $\nu^*(L)\in W^4_{14}(C)$, which has the property
$h^0(C, \nu^*L\otimes \OO_C(-y-q))=4$ (since $W^4_{12}(C)=\emptyset$, there exists a section of $L$ that does not vanish simultaneously at both $y$ and $q$).
However, the bundles of
type $A\otimes \OO_C(y)$ or $A\otimes \OO_C(q)$ with $A\in
W^4_{13}(C)$, do not appear in this association, though $(y, A\otimes \OO_C(y)), (y, A\otimes \OO_C(q))\in Y$. In fact, they correspond to the situation when
$L\in \overline{W}_{14}^4(C_{yq})$ is not locally free, in which case
necessarily $L=\nu_*(A)$ for some $A\in W^4_{13}(C)$. Thus, for a point $y\in C-\{q\}$, there is a birational morphism
$\pi_1^{-1}(y)\rightarrow \overline{W}_{14}^4(C_{yq})$ which is an isomorphism over the locus of locally free sheaves. More precisely, $\overline{W}_{14}^4(C_{yq})$ is obtained from $\pi_1^{-1}(y)$ by identifying the disjoint divisors $\Gamma_1\cap \pi_1^{-1}(y)$ and $\Gamma_2\cap \pi_1^{-1}(y)$.

A special analysis is required
when $y=q$,  when $C_{yq}$ degenerates to $C\cup _q
E_{\infty}$, where $E_{\infty}$ is a rational nodal cubic. If
$\{l_C, l_{E_{\infty}}\}\in \sigma^{-1}([C\cup_{q} E_{\infty}])$,
then the corresponding Brill-Noether numbers with respect to $q$ satisfy $\rho(l_C, q)\geq 0$ and $\rho(l_{E_{\infty}}, q)\leq 2$. The statement about the restrictions
$\E_{| \sigma^*(C_0)}$ and $\F_{| \sigma^*(C_0)}$ follows, because both restrictions are defined by dropping the information coming from the elliptic tail.
\end{proof}
To describe $\sigma^*(C_1)\subset \widetilde{\mathfrak{G}}^4_{14}$, where $[C]\in \cM_{11}$, we define the determinantal $3$-fold
$$X:=\{(y, L)\in C\times W^4_{14}(C): h^0(L\otimes
\OO_C(-2y))=4\}.$$
In what follows we use notation from \cite{EH1}, to denote vanishing sequences of limit linear series:
\begin{proposition}\label{limitlin1}
With notation as above, the $3$-fold $X$ is an irreducible component of $\sigma^*(C_1)$. Moreover one has that
$c_3\bigl((\F-\mathrm{Sym}^2\E)_{| \sigma^*(C_1)}\bigr)=c_3\bigl((\F-\mathrm{Sym}^2\E)_{| X}\bigr)$.
\end{proposition}
\begin{proof} By the additivity of the Brill-Noether number, if
$\{l_C, l_E\}\in \sigma^{-1}([C\cup_y E])$, we
have that $2=\rho(12, 4, 14)\geq \rho(l_C, y)+\rho(l_E, y)$. Since
$\rho(l_E, y)\geq 0$, we obtain that $\rho(l_C, y)\leq 2$. If
$\rho(l_E, y)=0$, then $l_E=9y+|\OO_E(5y)|$, that is, $l_E$ is
uniquely determined, while the aspect $l_C\in G^4_{14}(C)$ is a complete
$\mathfrak g^4_{14}$ with a cusp at the variable point $y\in C$.
This gives rise to an element from $X$. The remaining components of $\sigma^*(C_1)$ are indexed by Schubert
indices $\bar{\alpha}:=(0\leq \alpha_0\leq \cdots \leq \alpha_4\leq 10)$ such that $\bar{\alpha}>(0, 1, 1, 1, 1)$ and $5\leq \sum_{j=0}^4 \alpha_j \leq 7$.  For such $\bar{\alpha}$, we set $\bar{\alpha}^c:=(10-\alpha_4,\ldots, 10-\alpha_0)$ to be the complementary Schubert index, then define $$X_{\bar{\alpha}}:=\{(y, l_C)\in C\times G^4_{14}(C): \alpha^{l_C}(y)\geq \bar{\alpha}\}\ \mbox{ and }Z_{\bar{\alpha}}:=\{l_E\in G^4_{14}(E): \alpha^{l_E}(y)\geq \bar{\alpha}^c\}.$$ Then $\sigma^*(C_1)=X+\sum_{\bar{\alpha}} X_{\bar{\alpha}}\times Z_{\bar{\alpha}}$. The last claim follows by dimension reasons. Since $\mbox{dim } X_{\bar{\alpha}}=1+\rho(11, 4, 14)-\sum_{j=0}^4 \alpha_j<3$, for every $\bar{\alpha}>(0, 1, 1, 1, 1)$ and the restrictions of both $\E$ and $\F$ are pulled-back from $X_{\bar{\alpha}}$, one obtains that $c_3(\F- \mbox{Sym}^2\E)_{| X_{\bar{\alpha}}\times Z_{\bar{\alpha}}}=0$.
\end{proof}

We also recall standard facts about intersection theory on
Jacobians. For a Brill-Noether general curve $[C]\in \cM_g$, we  denote
by $\P$ a Poincar\'e bundle on $C\times \mbox{Pic}^d(C)$ and by
$\pi_1:C\times \mbox{Pic}^d(C)\rightarrow C$ and $\pi_2:C\times
\mbox{Pic}^d(C)\rightarrow \mbox{Pic}^d(C)$ the projections. We
define the cohomology class $\eta=\pi_1^*([\mathrm{point}])\in H^2(C\times
\mbox{Pic}^d(C))$, and if $\delta_1,\ldots, \delta_{2g}\in H^1(C,
\mathbb Z)\cong H^1(\mbox{Pic}^d(C), \mathbb Z)$ is a symplectic
basis, then we set
$$\gamma:=-\sum_{\alpha=1}^g
\Bigl(\pi_1^*(\delta_{\alpha})\pi_2^*(\delta_{g+\alpha})-\pi_1^*(\delta_{g+\alpha})\pi_2^*(\delta_
{\alpha})\Bigr)\in H^2(C\times \mathrm{Pic}^d(C)).$$ One has the formula $c_1(\P)=d\eta+\gamma,$
corresponding to the Hodge decomposition of $c_1(\P)$, as well as the relations
$\gamma^3=0,\ \gamma \eta=0,\  \eta^2=0$ and
$\gamma^2=-2\eta \pi_2^*(\theta).$
On $W^r_d(C)$ there is a
tautological rank $r+1$ vector bundle
$\mathcal{M}:=(\pi_2)_{*}(\mathcal{P}_{| C\times W^r_d(C)})$. To compute the
Chern numbers of $\mathcal{M}$ we employ the Harris-Tu
formula \cite{HT}. We write $$\sum_{i=0}^r
c_i(\mathcal{M}^{\vee})=(1+x_1)\cdots (1+x_{r+1}),$$ and then for every
class $\zeta \in H^*(\mbox{Pic}^d(C), \mathbb Z)$ one has the
following formula:
\begin{equation}\label{harristu}
x_1^{i_1}\cdots x_{r+1}^{i_{r+1}}\
\zeta=\mbox{det}\Bigl(\frac{\theta^{g+r-d+i_j-j+l}}{(g+r-d+i_j-j+l)!}\Bigr)_{1\leq
j, l\leq r+1}\ \zeta.
\end{equation}

We compute the classes of the $3$-folds that appear in Propositions \ref{limitlin0} and \ref{limitlin1}:
\begin{proposition}\label{xy}
Let $[C, q]\in \cM_{11, 1}$ be a Brill-Noether general pointed curve. If $\cM$ denotes the tautological rank $5$ vector
bundle over $W^4_{14}(C)$ and $c_i:=c_i(\cM^{\vee})\in H^{2i}(W^4_{14}(C), \mathbb C)$, then one has
the following relations:
\begin{enumerate}
\item
$[X]=\pi_2^*(c_4)-6\eta \theta\pi_2^*(c_2)+(48\eta+2\gamma) \pi_2^*(c_3) \in
H^8(C\times W^4_{14}(C), \mathbb C)$.
\item
$[Y]=\pi_2^*(c_4)-2\eta \theta \pi_2^*(c_2) + (13\eta+\gamma) \pi_2^*(c_3) \in
H^8(C\times W^4_{14}(C), \mathbb C)$.
\end{enumerate}
\end{proposition}
\begin{proof}
We start by noting that $W^4_{14}(C)$ is a smooth $6$-fold isomorphic to  the symmetric product $C_6$. We realize $X$ as the degeneracy locus of a vector
bundle morphism defined over $C\times W^4_{14}(C)$. For each pair $(y, L)\in C\times
W^4_{14}(C),$ there is a natural map $$H^0(C, L\otimes
\OO_{2y})^{\vee}\rightarrow H^0(C, L)^{\vee}$$ which globalizes to a
vector bundle morphism $\zeta: J_1(\mathcal{P})^{\vee} \rightarrow
\pi_2^*(\cM)^{\vee}$ over $C\times W^4_{14}(C)$. Then we have the identification
$X=Z_1(\zeta)$ and the Thom-Porteous formula gives that
$[X]=c_4\bigl(\pi^*_2(\cM)- J_1(\mathcal{P}^{\vee})\bigr).$ From the
usual exact sequence over $C\times \mbox{Pic}^{14}(C)$
$$
0\longrightarrow \pi_1^*(K_C)\otimes \mathcal{P} \longrightarrow
J_1(\mathcal{P}) \longrightarrow \mathcal{P} \longrightarrow 0, $$
we can compute the total Chern class of the jet bundle
$$c_t(J_1(\mathcal{P})^{\vee})^{-1}=\Bigl(\sum_{j\geq
0}(d(L)\eta+\gamma)^j\Bigr)\cdot \Bigl(\sum_{j\geq
0}\bigl((2g(C)-2+d(L))\eta+\gamma\bigr)^j\Bigr)=1-6\eta \theta +48\eta +2\gamma, $$ which
quickly leads to the formula for $[X]$. To compute $[Y]$ we proceed
in a similar way. We denote by $\mu, \nu:C\times C\times
\mbox{Pic}^{14}(C)\rightarrow C\times \mbox{Pic}^{14}(C)$ the two
projections, by $\Delta\subset C\times C\times \mbox{Pic}^{14}(C)$
the diagonal  and we set $\Gamma_q:=\{q\}\times \mbox{Pic}^{14}(C)$.
We introduce the rank $2$ vector bundle
$\cB:=(\mu)_*\bigl(\nu^*(\mathcal{P})\otimes
\OO_{\Delta+\nu^*(\Gamma_q)}\bigr)$ defined over $C\times
W^4_{14}(C)$. We note that there is a bundle morphism $\chi:
\cB^{\vee}\rightarrow (\pi_2)^*(\cM)^{\vee}$, such that
$Y=Z_1(\chi)$. Since we also have that
$$c_t(\cB^{\vee})^{-1}=\bigl(1+(d(L)\eta+\gamma)+(d(L)\eta+\gamma)^2+\cdots\bigr)\bigl(1-\eta\bigr),$$
we immediately obtained the stated expression for $[Y]$.
\end{proof}
\begin{proposition}\label{a121}
Let $[C]\in \cM_{11}$ and denote by $\mu, \nu:C\times C\times
\mathrm{Pic}^{14}(C)\rightarrow C \times \mathrm{Pic}^{14}(C)$ the
natural projections. We define the vector bundles $\cA_2$ and $\cB_2$ on
$C\times \rm{Pic}^{14}(C)$ having fibres $$\cA_2(y,
L)=H^0(C, L^{\otimes 2}\otimes
\OO_C(-2y)) \ \mbox{ and } \ \cB_2(y, L)=H^0(C, L^{\otimes 2}\otimes  \OO_C(-y-q)),$$ respectively.  One has the following formulas:
$$ c_1(\cA_2)=-4\theta-4\gamma-76\eta \ \mbox{  } \ c_1(\cB_2)=-4\theta-2\gamma-27\eta,$$
$$c_2(\cA_2)=8\theta^2+280 \eta \theta+16 \gamma \theta, \mbox{  }  \ c_2(\cB_2)=8\theta^2+100\eta\theta+8\theta \gamma,$$
$$c_3(\cA_2)=-\frac{32}{3}\theta^3-512 \eta\theta^2-32\theta^2\gamma \ \mbox{ and  } \ c_3(\cB_2)=-\frac{32}{3}\theta^3-184\eta\theta^2-16\theta^2\gamma.$$
\end{proposition}
\begin{proof} Immediate application of Grothendieck-Riemann-Roch with respect to $\nu$.
\end{proof}

Before our next result, we recall that if $\mathcal{V}$ is a vector bundle of rank $r+1$ on a
variety $X$, we have the formulas:
\begin{enumerate}
\item $c_1(\mathrm{Sym}^2  (\mathcal{V}))=(r+2) c_1(\mathcal{V})$.
\item $c_2(\mathrm{Sym}^2 (\mathcal{V}))=\frac{r(r+3)}{2} c_1^2(\mathcal{V})+(r+3)c_2(\mathcal{V})$.
\item $c_3(\mathrm{Sym}^2 (\mathcal{V}))=\frac{r(r+4)(r-1)}{6} c_1^3(\mathcal{V})+(r+5)c_3(\mathcal{V})+(r^2+4r-1)c_1(\mathcal{V}) c_2(\mathcal{V})$.
\end{enumerate}

We expand $\sigma_*\bigl(c_3(\F-\mbox{Sym}^2\E)\bigr)\equiv a\lambda-b_0\delta_0-b_1\delta_1\in A^1(\cM_{12}^p)$ and determine the coefficients
$a, b_0$ and $b_1$. This will suffice in order to compute $s(\overline{\mathfrak{D}}_{12})$.

\begin{theorem}\label{d1}
Let $[C]\in \cM_{11}$ be a Brill-Noether general curve  and denote by
$C_1\subset \Delta_1\subset \mm_{12}$ the associated test curve. Then
the coefficient of $\delta_1$ in the expansion of
$\overline{\mathfrak{D}}_{22}$ is equal to $$b_1=\frac{1}{2g(C)-2} \sigma^*(C_1)\cdot c_3\bigr(\F- \mathrm{Sym}^2\E\bigr)=9867.$$
\end{theorem}

\begin{proof} We intersect the degeneracy locus of the map
$\phi:\mbox{Sym}^2(\E)\rightarrow \F$ with the $3$-fold $\sigma^*(C_1)=X+\sum_{\bar{\alpha}} X_{\bar{\alpha}}\times Z_{\bar{\alpha}}$.
As  already explained in Proposition \ref{limitlin1}, it is enough to estimate the contribution coming from $X$ and we can write
$$\sigma^*(C_1)\cdot
c_3(\F-\mbox{Sym}^2\E)=c_3(\F_{|X})-c_3(\mbox{Sym}^2 \E_{| X})-c_1(\F_{|X})c_2(\mbox{Sym}^2
\E_{| X})+$$
$$+2c_1(\mbox{Sym}^2\E_{|X})c_2(\mbox{Sym}^2\E_{| X})-c_1(\mbox{Sym}^2 \E_{| X})c_2(\F_{|X})+c_1^2(\mbox{Sym}^2 \E_{| X})c_1(\F_{|X})-c_1^3(\mbox{Sym}^2 \E_{|X}).$$
We are going to compute each term in the right-hand-side of this
expression.

Recall that we have constructed in Proposition \ref{xy} a vector
bundle morphism $\zeta: J_1(\P)^{\vee}\rightarrow
\pi_2^*(\cM)^{\vee}$. We consider the kernel line bundle $\mbox{Ker}(\zeta)$.
If $U$ is the  line bundle on $X$ with  fibre $$U(y,
L)=\frac{H^0(C, L)}{H^0(C, L\otimes\OO_C(-2y))}\hookrightarrow H^0(C, L\otimes \OO_{2y})$$ over a
point $(y, L)\in X$, then one has an exact sequence over $X$
$$0\rightarrow U\rightarrow J_1(\P)\rightarrow \mbox{Ker}(\zeta)^{\vee}\rightarrow 0.$$
In particular, $c_1(U)=2\gamma+48\eta-c_1(\mathrm{Ker}(\zeta))^{\vee}$.
The products of the Chern class of $\mbox{Ker}(\zeta)^{\vee}$ with other classes on $C\times W^4_{14}(C)$ can be computed from the
Harris-Tu formula  \cite{HT}:
\begin{equation}\label{harristu}
c_1(\mathrm{Ker}(\zeta)^{\vee})\cdot \xi_{| X}=-c_5(\pi_2^*(\cM)^{\vee}-J_1(\P)^{\vee})\cdot \xi_{|X}=-\bigl(\pi_2^*(c_5)-6\eta \theta \pi_2^*(c_3)
+(48\eta+2\gamma)\pi_2^*(c_4)\bigr)\cdot \xi_{|X},
\end{equation} for any class
$\xi\in H^2(C\times W^4_{14}(C), \mathbb C)$.

If $\cA_3$ denotes the rank $18$ vector bundle on $X$ having fibres
$\cA_3(y, L)=H^0(C, L^{\otimes 2})$, then there is an injective
 morphism $U^{\otimes 2}\hookrightarrow \cA_3/\cA_2$, and we
consider the quotient sheaf $$\G:=\frac{\cA_3/\cA_2}{U^{\otimes
2}}.$$ Since the morphism $U^{\otimes 2}\rightarrow
\cA_3/\cA_2$ vanishes along the locus of
pairs $(y, L)$ where $L$ has a base point, $\G$ has torsion along
$\Gamma\subset X$. A straightforward local analysis now shows that
$\F_{|X}$ can be identified as a subsheaf of $\cA_3$ with the kernel
of the map $\cA_3\rightarrow \G$. Therefore, there is an exact
sequence of vector bundles on $X$
$$0\rightarrow \cA_{2 |X}\rightarrow \F_{| X}\rightarrow
U^{\otimes 2}\rightarrow 0,$$ which over a general point of $X$
corresponds to the decomposition
$$\F(y, L)=H^0(C, L^{\otimes 2}\otimes \OO_C(-2y))\oplus \mathbb
C\cdot u^2,$$ where $u \in H^0(C, L)$ is such that
$\mbox{ord}_y(u)=1$. The analysis above, shows that the sequence stays exact over the curve $\Gamma$ as well.  Hence $$c_1(\F_{| X})=c_1(\cA_{2 |X})+2c_1(U), \ c_2(\F_{| X})=c_2(\cA_{2 |X})+2c_1(\cA_{2 | X}) c_1(U) \ \mbox{ } \ \mbox{ and }$$
$$c_3(\F_{|X})=c_3(\cA_2)+2c_2(\cA_{2| X}) c_1(U).$$ Furthermore, since $\E_{| X}=\pi_2^*(\cM)_{|X}$, we obtain that:
$$
\sigma^*(C_1)\cdot c_3\bigl(\F-\mbox{Sym}^2\E\bigr)=c_3(\cA_{2
|X})+c_2(\cA_{2 |X})c_1(U^{\otimes 2})-c_3(\mathrm{Sym}^2 \pi_2^*\cM_{| X})-$$
$$-\Bigl(\frac{r(r+3)}{2}c_1(\pi_2^*\cM_{|X})+(r+3)c_2(\pi_2^*\cM_{|X})\Bigr)\cdot \Bigl(c_1(\cA_{2
|X})+c_1(U^{\otimes 2})-2(r+2)c_1(\pi_2^*\cM_{|X})\Bigr)-$$
$$-(r+2)c_1(\pi_2^*\cM_{|X}) c_2(\cA_{2 |X})-(r+2)c_1(\pi_2^*\cM_{|X}) c_1(\cA_{2 | X}) c_1(U^{\otimes 2})+$$
$$+(r+2)^2 c_1^2(\pi_2^*\cM_{| X}) c_1(\cA_{2 | X})+(r+2)^2c_1^2(\pi_2^*\cM_{|X})c_1(U^{\otimes 2})-(r+2)^3 c_1^3(\pi_2^*\cM_{| X}).$$
As before, $c_i(\pi_2^*\cM_{|X}^{\vee})=\pi_2^*(c_i)\in H^{2i}(X, \mathbb C)$. The coefficient of $c_1(\mathrm{Ker}(\zeta)^{\vee})$ in the product
$\sigma^*(C_1)\cdot c_3\bigl(\F-\mbox{Sym}^2 \E\bigr)$
is evaluated via (\ref{harristu}). The part of this product that does not contain
$c_1(\mathrm{Ker}(\zeta)^{\vee})$
equals
$$\small{28\pi_2^*(c_2)\theta-88\pi_2^*(c_1^2)\theta+440\eta\pi_2^*(c_1^2)-53\pi_2^*(c_1c_2)-\frac{32}{3}\theta^3+128\eta\theta^2-
432\eta\theta\pi_2^*(c_1)}$$
$$\small{+64\pi_2^*(c_1^3)-140\eta\pi_2^*(c_2)+48\theta^2\pi_2^*(c_1)+9\pi_2^*(c_3)\in H^6(C\times W^4_{14}(C), \mathbb C).}$$
Multiplying this quantity by the class $[X]$ obtained in Proposition \ref{xy} and then adding to it the contribution coming from $c_1(\mbox{Ker}(\zeta)^{\vee})$, one obtains a homogeneous polynomial of degree $7$ in $\eta, \theta$ and $\pi_2^*(c_i)$ for $1\leq i\leq 4$. The only non-zero monomials are those containing $\eta$. After retaining only these monomials, the resulting degree $6$ polynomial in $\theta, c_i\in H^*(W^4_{14}(C), \mathbb Z)$ can be brought to a manageable form, by noting  that, since $h^1(C, L)=1$, the classes $c_i$ are not independent. Precisely, if one fixes a divisor $D\in C_e$ of large degree,
there is an exact sequence
$$0\rightarrow \cM\rightarrow (\pi_{2})_*\bigl(\P\otimes
\OO(\pi^*D)\bigr)\rightarrow (\pi_2)_*\bigl(\P\otimes
\OO(\pi_1^*D)_{| \pi_1^*D}\bigr)\rightarrow R^1\pi_{2
*}\bigl(\P_{| C\times W^4_{14}(C)}\bigr)\rightarrow 0,$$
from which, via the well-known fact  $c_t\bigl((\pi_2)_*(\P \otimes
\OO(\pi_1^*D))\bigr)=e^{-\theta}$, it follows that
$$c_{t}R^1\pi_{2
*}\bigl(\P_{| C\times W^4_{14}(C)}\bigr)\cdot e^{-\theta}=\sum_{i=0}^4 (-1)^i c_i.$$
Hence
$c_{i+1}=\theta^ic_i/i!-i\theta^{i+1}/(i+1)!$, for all $i\geq 2$.
After routine manipulations, one finds that
$b_1= \sigma^*(C_1)\cdot
c_3(\F-\mbox{Sym}^2(\E))/20= 9867.$
\end{proof}

\begin{theorem}\label{d0}
Let $[C, q]\in \cM_{11, 1}$ be a Brill-Noether general pointed curve
and we denote by $C_0\subset \Delta_0\subset \mm_{12}$ the associated test curve. Then
$\sigma^*(C_0)\cdot c_3(\F-\mathrm{Sym}^2\E)=22b_0-b_1=32505$.
It follows that $b_0=1926$.
\end{theorem}

\begin{proof} As already noted in Proposition \ref{limitlin0}, the vector bundles
$\E_{| \sigma^*(C_0)}$ and $\F_{| \sigma^*(C_0)}$ are both
pull-backs of vector bundles on $Y$ and we denote these
vector bundles $\E$ and $\F$ as well, that is, $\E_{| \sigma^*(C^0)}=(\epsilon\circ f)^*(\E_{|
Y})$ and $\F_{| \sigma^*(C_0)}=(\epsilon\circ f)^*(\F_{| Y})$. Like in the proof of Theorem \ref{d1}, we evaluate each term
appearing in
$\sigma^*(C_0)\cdot c_3(\F-\mbox{Sym}^2(\E))$.

Let $V$ be the  line bundle on $Y$ with  fibre $$V(y,
L)=\frac{H^0(C, L)}{H^0(C, L\otimes\OO_C(-y-q))}\hookrightarrow H^0(C, L\otimes \OO_{y+q})$$ over a
point $(y, L)\in Y$. There is an exact sequence of vector bundles over $Y$
$$0\longrightarrow V\longrightarrow \cB \longrightarrow \mbox{Ker}(\chi)^{\vee}\longrightarrow 0,$$
where $\chi:
\cB^{\vee}\rightarrow \pi_2^*(\cM)^{\vee}$ is the bundle
 morphism defined in the second part of Proposition \ref{xy}. In particular, $c_1(V)=13\eta+\gamma-c_1(\mathrm{Ker}(\chi^{\vee})$.
By using again \cite{HT}, we find the following formulas for the Chern
numbers of $\mathrm{Ker}(\chi)^{\vee}$:
$$c_1(\mathrm{Ker}(\chi)^{\vee})\cdot \xi_{|Y}=-c_5\bigl(\pi_2^*(\cM)^{\vee}-\cB^{\vee}\bigr)\cdot \xi_{| Y}=-(\pi_2^*(c_5)+\pi_2^*(c_4)(13\eta+\gamma)-2\pi_2^*(c_3)\eta \theta)\cdot \xi_{| Y},$$
for any class $\xi\in H^2(C\times W^4_{14}(C), \mathbb C)$.
Recall that we introduced the vector bundle $\cB_2$ over
$C\times W^4_{14}(C)$ with fibre $\cB_2(y, L)=H^0(C, L^{\otimes
2}\otimes \OO_C(-y-q))$. We claim that one has an exact sequence of
bundles over $Y$
\begin{equation} \label{exseq} 0\longrightarrow
\cB_{2 |Y}\longrightarrow \F_{| Y}\longrightarrow V^{\otimes
2}\longrightarrow 0. \end{equation}
 If $\cB_3$ is the
vector bundle on $Y$ with fibres
$\cB_3(y, L)=H^0(C, L^{\otimes 2})$,  we have an injective morphism of sheaves
$V^{\otimes 2} \hookrightarrow \cB_3/\cB_2$ locally given by
$$v^{\otimes 2}\mapsto v^2 \mbox{ mod } H^0(C, L^{\otimes 2}\otimes
\OO_C(-y-q)),$$ where $v\in H^0(C, L)$ is any section not vanishing
at $q$ and $y$. Then $\F_{|Y}$ is canonically identified with the
kernel of the projection morphism
$$\cB_3\rightarrow \frac{\cB_3/\cB_2}{V^{\otimes 2}}$$
and the exact sequence (\ref{exseq}) now becomes clear. Therefore
$c_1(\F_{| Y})=c_1(\cB_{2 |Y})+2c_1(V)$, $c_2(\F_{|Y})=c_2(\cB_{2
|Y})+2c_1(\cB_{2| Y}) c_1(V)$ and $c_3(\F_{|Y})=c_3(\cB_{2| Y})+2c_2(\cB_{2 |Y}) c_1(V)$.  The part of the total intersection
number
$\sigma^*(C_0)\cdot c_3(\F-\mbox{Sym}^2(\E))$ that does not contain $c_1(\mbox{Ker}(\chi^{\vee}))$ equals $$28\pi_2^*(c_2)\theta-88\pi_2^*(c_1^2)\theta-22\eta\pi_2^*(c_1^2)-53\pi_2^*(c_1c_2)-\frac{32}{3}\theta^3+$$
$$-8\eta\theta^2+24\eta\theta\pi_2^*(c_1)+64\pi_2^*(c_1^3)+7\eta\pi_2^*(c_2)+48\theta^2\pi_2^*(c_1)+9\pi_2^*(c_3)\in H^6(C\times W^4_{14}(C), \mathbb C)$$
and this gets multiplied with the class $[Y]$ from Proposition \ref{xy}. The coefficient of $c_1(\mathrm{Ker}(\zeta)^{\vee})$ in $\sigma^*(C_0)\cdot c_3\bigl(\F-\mathrm{Sym}^2\E\bigr)$ equals
$$-2c_2(\cB_{2| Y})-2(r+2)^2\pi_2^*(c_1^2)-2(r+2)c_1(\cB_{2 |Y})\pi_2^*(c_1)+r(r+3)\pi_2^*(c_1^2)+2(r+3)\pi_2^*(c_2).$$ All in all,
$22b_0-b_1=\sigma^*(C_0)\cdot c_3(\F-\mathrm{Sym}^2\E)$
and we evaluate this using  (\ref{harristu}).
\end{proof}
The following result follows from the definition of the vector bundles $\E$ and $\F$ given in Proposition \ref{vectbundles}:
\begin{theorem}\label{elltail}
Let $[C, q]\in \cM_{11, 1}$ be a Brill-Noether general pointed curve and
$R\subset \mm_{12}$ the pencil obtained by attaching at the fixed point $q\in C$ a pencil of plane cubics. Then
$$a-12b_0+b_1=\sigma_* c_3\bigl(\F-\mathrm{Sym}^2\E\bigr)\cdot R=0.$$
\end{theorem}

\noindent
\emph{End of the proof of Theorem \ref{m12}}. First we note that the virtual divisor $\mathfrak{D}_{12}$ is a genuine divisor on $\cM_{12}$. Assuming by contradiction that for every curve $[C]\in \cM_{12}$, there exists $L\in W^4_{14}(C)$ such that $\mu_0(L)$ is not-injective, one can construct a stable vector bundle $E$ of rank $2$ sitting in an extension  $$0\longrightarrow K_C\otimes L^{\vee}\longrightarrow E\longrightarrow L\longrightarrow 0,$$ such that $h^0(C, E)=h^0(C, L)+h^1(C, L)=7$, and for which the Mukai-Petri map $\mbox{Sym}^2 H^0(C, E)\rightarrow H^0(C, \mbox{Sym}^2 E)$ is not injective. This contradicts the main result from \cite{T}.
To determine the slope of $\overline{\mathcal{D}}_{12}$, we write
$\overline{\mathfrak{D}}_{12}\equiv a \lambda-\sum_{j=0}^{6} b_j
\delta_j\in \mbox{Pic}(\mm_{12})$. Since $\frac{a}{b_0}=\frac{4415}{642}\leq \frac{71}{10}$, we can
apply Corollary 1.2 from \cite{FP}, which gives the inequalities
$b_j\geq b_0$ for $1\leq j\leq 6$. Therefore
$s(\overline{\mathfrak{D}}_{12})=\frac{a}{b_0}<6+\frac{12}{13}$.\hfill $\Box$

\vskip 4pt

We close by discussing a second counterexample to the Slope Conjecture on $\mm_{12}$.
\begin{definition}
Let $V$ be a vector space. A pencil of quadrics $\ell\subset \PP(\mbox{Sym}^2(V))$ is said to be \emph{degenerate} if the intersection of $\ell$ with the discriminant divisor $\DD(V)\subset \PP(\mbox{Sym}^2(V))$ is non-reduced.
\end{definition}

 A general curve $[C]\in \cM_{12}$ has finitely many linear systems
$A\in W^5_{15}(C)$. As a consequence of the maximal rank conjecture \cite{Vo}, the multiplication map
$$\mu_0(A):\mbox{Sym}^2 H^0(C, A)\rightarrow H^0(C, A^{\otimes 2})$$ is surjective for each $A\in W^5_{15}(C)$, in particular $\PP_{C,A}:=\PP\bigl(\mbox{Ker } \mu_0(A)\bigr)$ is a pencil of quadrics
in $\PP^5$ containing the image of the map $C\stackrel{|A|}\longrightarrow \PP^5$. One expects that the pencil $\PP_{C,A}$ to be non-degenerate. By imposing the condition that it be degenerate, we produce a divisor on $\cM_{12}$, whose class we compute.

We shall make essential use of the following result \cite{FR}. Let $X$ be a smooth projective variety, $\E$ and $\F$  vector bundles on $X$ with $\mbox{rk}(\E)=e$ and $\mbox{rk}(\F)={e+1\choose 2}-2$,
and $\varphi:\mbox{Sym}^2(\E)\rightarrow \F$ a surjective vector bundle morphism. Then the class of the locus
$$\H:=\Bigl\{x\in X: \PP\bigl(\mbox{Ker}\ \varphi(x)\bigr)\subset \PP\bigl(\mbox{Sym}^2 \E(x)\bigr) \ \mbox{ is a degenerate pencil}\Bigr\},$$
assuming it is of codimension one in $X$, is equal to \begin{equation}\label{rima}
[\H]=(e-1)\bigl(e\ c_1(\F)-(e^2+e-4)c_1(\E)\bigr)\in A^1(X).
\end{equation}

\begin{theorem}\label{sec12}
The locus consisting of smooth curves of genus $12$
$$\H_{12}:=\bigl\{[C]\in \cM_{12}: \PP_{C, A} \ \mbox{ is degenerate for a } A\in W^5_{15}(C)\bigr\}$$
is an effective divisor. The slope of its closure $\hh_{12}$ inside $\mm_{12}$ equals $s(\hh_{12})=\frac{373}{54}<6+\frac{12}{13}$.
\end{theorem}

\begin{proof} We only sketch the main steps. We retain the notation in the proof of Theorem \ref{m12} and consider the stack $\sigma:\widetilde{\mathfrak{G}}^5_{15}\rightarrow \ttem_{12}$ of limit linear series of type $\mathfrak g^5_{15}$. Using \cite{F2} Proposition 2.8, there exist two vector bundles $\E$ and $\F$ over $\widetilde{\mathfrak{G}}^5_{15}$
together with a morphism $\varphi:\mbox{Sym}^2(\E)\rightarrow \F$, such that over a point $[C, A]\in \sigma^{-1}(\cM_{12}^p)$ corresponding to a smooth underlying curve one has the description of its fibres
$\E(C, A)=H^0(C, A)$ and $\F(C,A)=H^0(C, A^{\otimes 2})$. Moreover $\varphi(C,A)$ is the multiplication map $\mu_0(A)$. The extension of $\E$ and $\F$ over the boundary of $\widetilde{\mathfrak{G}}^5_{15}$ is identical to the one appearing in Proposition \ref{vectbundles}. Applying (\ref{rima}), the class of the restriction $\widetilde{\H}_{12}:=\hh_{12}\cap \cM_{12}^p$ is equal to
$$[\widetilde{\H}_{12}]^{virt}=10\sigma_*\bigl(6c_1(\F)-38c_1(\E)\bigr)\in A^1(\ttem_{12}).$$
The pushforward classes $\sigma_*(c_1(\E))$ and $\sigma_*(c_1(\F))$ can be determined following \cite{F2} Propositions 2.12 and 2.13, which after manipulations leads to the claimed slope.

To prove that $\H_{12}$ is indeed a divisor, note first that $\mathfrak{G}^5_{15}$ being isomorphic to the Hurwitz space $\mathfrak{G}^1_7$ is irreducible. To establish that for a general curve $[C]\in \cM_{12}$, the pencil $\PP_{C,A}$ is non-degenerate for all linear systems $A\in W^5_{15}(C)$, it suffices to produce \emph{one example} of a smooth curve $C\subset \PP^{5}$ with $g(C)=12$ and $\mbox{deg}(C)=15$, with $\PP_{C, \OO_C(1)}$ non-degenerate. This is carried out via the use of \emph{Macaulay} in a way similar to the proof of Theorem 2.7 in \cite{F1} for a curve $C$  lying on a particular rational surface in $\PP^5$.
 \end{proof}

\end{document}